\documentclass[a4paper,11pt]{amsart}
\usepackage[T1]{fontenc}

\usepackage[utf8]{inputenc}

\usepackage[stretch=10]{microtype}
\usepackage{mathpazo}

\usepackage{hyphenat} %
\usepackage[all,cmtip,2cell]{xy}
\UseTwocells
\xyoption{2cell}
\usepackage{amsfonts}
\usepackage{amsthm}
\usepackage{amsmath}
\usepackage{amssymb}
\usepackage[mathscr]{euscript}
\usepackage{enumerate}
\usepackage{verbatim} %
\usepackage{mathtools}
\usepackage{url}
\usepackage{tikz-cd}
\usepackage{todonotes}
\usepackage{geometry}
\usepackage[pdftex]{hyperref} %

\usepackage[nameinlink,capitalise,noabbrev]{cleveref}
\theoremstyle{definition} \newtheorem{defn}[equation]{Definition}
\theoremstyle{definition} \newtheorem{prop-def}[equation]{Proposition-Definition}
\theoremstyle{plain} \newtheorem{thm}[equation]{Theorem}
\theoremstyle{plain} \newtheorem{lemma}[equation]{Lemma}
\theoremstyle{plain} \newtheorem{prop}[equation]{Proposition}
\theoremstyle{plain} \newtheorem{cor}[equation]{Corollary}
\theoremstyle{plain} \newtheorem*{thm*}{Theorem}
\theoremstyle{remark} \newtheorem{rmk}[equation]{Remark}
\theoremstyle{definition} \newtheorem{notn}[equation]{Notation}
\theoremstyle{remark} \newtheorem{ex}[equation]{Example}
\numberwithin{equation}{section}

\newtheoremstyle{TheoremNum}
{}{}              %
{\itshape}                      %
{}                              %
{\bfseries}                     %
{.}                             %
{ }                             %
{\thmname{#1}\thmnote{ \bfseries #3}}%
\theoremstyle{TheoremNum}
\newtheorem{thmn}{Theorem}

\newcommand{\D}{\mathscr{D}}
\newcommand{\E}{\mathscr{E}}
\newcommand{\M}{\mathcal{M}}

\newcommand{\C}{\mathscr{C}}

\newcommand{\Z}{\mathbb{Z}}
\newcommand{\Q}{\mathbb{Q}}
\newcommand{\N}{\mathbb{N}}

\renewcommand{\S}{\mathscr{S}}

\newcommand{\bS}{\mathbb{S}}
\newcommand{\id}{\mathrm{id}}
\newcommand{\Ext}{\mathrm{Ext}}
\newcommand{\fin}{\mathrm{fin}}
\newcommand{\colim}{\mathrm{colim}}
\newcommand{\fib}{\mathrm{fib}}
\newcommand{\cofib}{\mathrm{cofib}}
\newcommand{\cof}{\mathrm{cofib}}

\newcommand{\Map}{\mathrm{Map}}

\newcommand{\Hom}{\mathrm{Hom}}
\newcommand{\adj}[4]{\xymatrix{#1 \ar@<.2pc>[r]^-{#3} & \ar@<.2pc>[l]^-{#4} #2 }}
\newcommand{\eqv}[4]{\xymatrix{#1 \ar@<.5pc>[r]^-{#3} & \ar@<.5pc>[l]^-{#4}_-\sim #2 }}

\renewcommand{\Bar}{\mathrm{Bar}}

\newcommand{\CAlg}{\mathrm{CAlg}}
\newcommand{\nuca}{\mathrm{NUCA}}
\newcommand{\Sp}{\mathrm{Sp}}

\newcommand{\Spc}{\mathrm{Sp}^{\textup{cn}}}
\newcommand{\Pic}{\mathrm{Pic}}

\newcommand{\Mod}{\mathrm{Mod}}

\newcommand{\Fun}{\mathrm{Fun}}

\newcommand{\sma}{\wedge}
\newcommand{\Sym}{\mathrm{Sym}}
\newcommand{\Der}{\mathrm{Der}}
\newcommand{\diag}{\mathrm{diag}}

\DeclareMathOperator{\otilde}{\widetilde{\otimes}}

\newcommand{\Exc}{\mathrm{Exc}}
\newcommand{\LO}{\mathbb{L}\Omega}
\renewcommand{\epsilon}{\varepsilon}

\newcommand{\paralelas}[2]{\ar@<.2pc>[r]^-{#1} \ar@<-.2pc>[r]_-{#2}}

\newcommand{\under}[2]{{{#1}/ \hspace{-.08cm} / {#2}}}
\newcommand{\AB}{{\under{A}{B}}}
\newcommand{\BB}{{\under{B}{B}}}

\newcommand{\RS}{{\under{R}{S}}}

\newcommand{\Mon}{\mathrm{Mon}_{E_\infty}}
\newcommand{\Grp}{\mathrm{Grp}_{E_\infty}}
\newcommand{\Grpn}{\mathrm{Grp}_{E_n}}
\newcommand{\Grpp}[1]{\mathrm{Grp}_{E_{#1}}}

\DeclareMathAlphabet{\mathbbe}{U}{bbold}{m}{n}

\definecolor{bubblegum}{rgb}{0.99, 0.76, 0.8}

\DeclareUnicodeCharacter{00A0}{ } %

\definecolor{dark-red}{rgb}{0.6,.15,.15}
\definecolor{dark-blue}{rgb}{.1,.1,.65}
\definecolor{dark-green}{rgb}{.1,.65,.1}
\hypersetup{pdfauthor={Nima Rasekh, Bruno Stonek},colorlinks=true, citecolor=dark-blue, urlcolor=dark-green, linkcolor=dark-red,breaklinks=true,hypertexnames=false}
\title{The cotangent complex and Thom spectra}
\author{Nima Rasekh}
\address{{\'E}cole Polytechnique F{\'e}d{\'e}rale de Lausanne, SV BMI UPHESS, Station 8, CH-1015 Lausanne, Switzerland}
\email{nima.rasekh@epfl.ch}
\author{Bruno Stonek}
\address{Institute of Mathematics, Polish Academy of Sciences,  ul. {\' S}niadeckich 8, 00-656 Warszawa, Poland}
\email{bruno@stonek.com}

\begin{document}
\date{\today}

\begin{abstract} 
The cotangent complex of a map of commutative rings is a central object in deformation theory. Since the 1990s, it has been generalized to the homotopical setting of $E_\infty$-ring spectra in various ways.

In this work we first establish, in the context of $\infty$-categories and using Goodwillie's calculus of functors, that various definitions of the cotangent complex of a map 
of $E_\infty$-ring spectra that exist in the literature are equivalent. We then turn our attention to a specific example. Let $R$ be an $E_\infty$-ring spectrum and $\Pic(R)$ denote its Picard $E_\infty$-group. Let $Mf$ denote the Thom $E_\infty$-$R$-algebra of a map of $E_\infty$-groups $f:G\to \Pic(R)$; examples of $Mf$ are given by various flavors of cobordism spectra. We prove that the cotangent complex of $R\to Mf$ is equivalent to the smash product of $Mf$ and the connective spectrum associated to $G$.
\end{abstract}

\maketitle

\section{Introduction}

\subsection{Deformation theory}

The cotangent complex arises as a central object of deformation theory. One wishes to understand extensions of functions between geometric spaces to infinitesimal thickenings of those spaces. Working algebraically, the simplest example of an infinitesimal thickening of a commutative $R$-algebra $S$ is given by a \emph{square-zero extension}: that is, a surjective map of commutative $R$-algebras $\phi:\widetilde S\to S$ such that the product of any two elements in the kernel of $\phi$ is zero.\\ %

One way to construct square-zero extensions is using derivations. If $M$ is an $S$-module, an \emph{$R$-linear derivation} $S\to M$ is a map of $R$-modules which satisfies the Leibniz condition. The $S$-module of $R$-linear derivations $\Der_R(S,M)$ is co-represented by the $S$-module of Kähler differentials $\Omega_{S/R}$. It is also represented by the commutative $R$-algebra over $S$ given by the \emph{trivial square-zero extension} $S\oplus M$ with multiplication given by $(s,m)(s',m')=(ss',sm'+ms')$. Therefore, we have bijections
\begin{equation}\label{intro-isos}
\Hom_{\Mod_S}(\Omega_{S/R},M) \cong \Der_R(S,M) \cong \Hom_{\CAlg_\RS}(S,S\oplus M)
\end{equation}
where $\CAlg_\RS$ denotes the category of commutative rings $C$ with maps of commutative rings $R\to C\to S$ composing to the unit of $S$.

Starting from a class in $\Ext^1_S(\Omega_{S/R},M)$ represented by an extension of $S$-modules \[M\to \widetilde{M}\to \Omega_{S/R},\] one can build a square-zero extension of $S$ by $M$ by pulling back the second map along the universal derivation $S \to \Omega_{S/R}$ in $\Mod_R$:
\[\xymatrix{\widetilde{S} \ar[r] \ar[d] & S \ar[d] \\ \widetilde{M} \ar[r] & \Omega_{S/R}.}\]
The pullback $\widetilde{S}$ gets a commutative $R$-algebra structure, and the map $\widetilde{S}\to S$ is surjective with kernel isomorphic to $M$ as an $R$-module. %

However, this does not produce all square-zero extensions: for that, one needs to \emph{derive} the module of Kähler differentials. The way this was originally achieved independently by André \cite{andre} and Quillen \cite{quillen-homology} was by simplicial methods. Quillen placed a model structure on the category of simplicial commutative algebras. This produces a simplicial $S$-module $\LO_{S/R}$ (called the \emph{algebraic cotangent complex} by Lurie \cite[25.3]{sag}), and the first \emph{André-Quillen cohomology $S$-module} $\Ext^1_S(\LO_{S/R},M)$ is isomorphic to the $S$-module of equivalence classes of square-zero extensions of $S$ by $M$.\\ %

We shall not take this approach, but rather work with the more general $E_\infty$-ring spectra. The module of Kähler differentials in this context is replaced by the \emph{cotangent complex}: if $A\to B$ is a map of $E_\infty$-ring spectra, the cotangent complex is a $B$-module $L_{B/A}$. Using $L_{S/R}$\footnote{When treating a discrete commutative ring as an $E_\infty$-ring spectrum, the Eilenberg--Mac Lane functor shall be understood.} instead of $\LO_{S/R}$, André-Quillen (co)homology is replaced by \emph{topological André-Quillen (co)homology}, also known as $TAQ$. The precise relation between $\LO_{B/A}$ and $L_{B/A}$ can be found in \cite[25.3.3.7/25.3.5.4]{sag}.

One can define square-zero extensions of $E_\infty$-ring spectra. The cotangent complex helps classify them. As is usual in homotopy theory, one does not merely want to describe \emph{equivalence classes} of square-zero extensions. One would like to prove that the whole $\infty$-category of square-zero extensions of an $E_\infty$-$A$-algebra $B$ is equivalent to the $\infty$-category of $A$-linear derivations $B\to \Sigma M$ where $M$ is a $B$-module.\footnote{Recall that $\Ext^1_B(L_{B/A},M)\simeq \pi_0 (\Map_{\Mod_B}(L_{B/A},\Sigma M))$.} %
This is proven in \cite[7.4.1.26]{ha}, provided one restricts the $\infty$-categories a bit. Note that the traditional definition of derivations as linear maps that satisfy the Leibniz rule does not work in this setup, whereas the interpretations of (\ref{intro-isos}) do.
\\

Apart from the connections to deformation theory, the cotangent complex $L_{B/A}$ helps detect useful properties of the map $f: A\to B$. For example, if $A$ and $B$ are connective, then $f$ is an equivalence if and only if $\pi_0(f):\pi_0(A)\to \pi_0(B)$ is an isomorphism and $L_{B/A}$ vanishes \cite[7.4.3.4]{ha}. The theory of the cotangent complex also helps in proving theorems that do not mention it at all: for example, if $A$ is an $E_\infty$-ring spectrum and a map of commutative rings $\pi_0(A)\to B_0$ is étale, then it lifts in an essentially unique way to an étale map of $E_\infty$-ring spectra $A\to B$ \cite[7.5.0.6]{ha}.

\subsection{Different approaches to the cotangent complex}

Let us trace the history of the $E_\infty$ cotangent complex of a map of $E_\infty$-ring spectra $A\to B$, since it has been defined in different ways in the literature.

The first definition can be found in a preprint by Kriz \cite{kriz}. It was defined as a sequential colimit built out of tensoring $B\sma_A B$ with spheres, in a certain way.

Basterra \cite{basterra} took another approach: she defined the cotangent complex to be the indecomposables of the \emph{augmentation ideal} $I(B\sma_A B)$ of $B\sma_A B$, which is the fiber of the multiplication map $B\sma_A B\to B$. She did not prove the equivalence with Kriz's approach.

Later, herself and Mandell \cite{basterra-mandell} established the connection between Basterra's definition of the cotangent complex and \emph{stabilization}. Just as Beck had observed in the sixties that the category of modules over a commutative ring $R$ was equivalently given by the abelian group objects in augmented commutative $R$-algebras \cite{beck}, they proved the $E_\infty$-analog. Abelian group objects have to be replaced by \emph{spectra objects}. They proved that to get $L_{B/A}$, one can start from $B\sma_AB$ considered as an augmented commutative $B$-algebra, then \emph{stabilize} it, i.e. apply $\Omega^\infty \Sigma^\infty$ to it, then take its augmentation ideal.

It was known to the experts that one could extract from the results of Basterra and Mandell an expression of the cotangent complex as a sequential colimit, similar to Kriz's expression, see e.g. \cite[Page 164]{schlichtkrull-higher}. However, we think a full description of how these approaches are connected has not appeared in the literature. We take the opportunity to expand on them in \cref{section-cotangent}.

The approach to the cotangent complex taken by Lurie in \cite[7.3]{ha} is closest to Basterra and Mandell's approach, albeit in the realm of $\infty$-categories rather than in that of model categories. We feel a unified study of the different approaches in Lurie's setting was lacking: we provide one here. We adopt the language of the Goodwillie calculus of functors as developed by Lurie in \cite[Chapter 6]{ha}. Our \cref{section-good} will swiftly introduce the necessary results. The fact that the the cotangent complex can be understood via Goodwillie calculus was known, see e.g. \cite[5.4]{kuhn-goodwillie}. 

In summary, we prove in the $\infty$-categorical context of \cite{ha} that the cotangent complex $L_{B/A}$ can be presented in the following ways:

\begin{itemize}
\item As the augmentation ideal of the stabilization of $B\sma_A B$, i.e. $I(\Omega^\infty \Sigma^\infty (B\sma_A B))$  (\ref{rmk-commuting}/\ref{eq-deftaq}),
\item As the excisive approximation of $I$ evaluated in $B\sma_AB$, i.e. $(P_1I)(B\sma_AB)$  (\ref{eq-p1}),
\item As the sequential colimit of $B$-modules (\ref{prop-colim})
\[L_{B/A}\simeq \colim_{\Mod_B} (\xymatrix@C+1pc{S^0\otilde_A B \ar[r] & \Omega (S^1\otilde_A B) \ar[r] & \Omega^2 (S^2\otilde_A B \ar[r]) & \cdots }), \]
\item As the module of indecomposables of the augmentation ideal of $B\sma_A B$ (\ref{cor-indecomposables}).
\end{itemize}
Here $- \tilde\otimes - $ is a certain operation to be introduced in \cref{not-otilde} that takes a pointed space and an $E_\infty$-$A$-algebra $B$ and returns a $B$-module.

\subsection{Thom spectra} The main result of this paper is the determination of the cotangent complex of \emph{Thom $E_\infty$-ring spectra}. Let us quickly recall what these are, following the $\infty$-categorical approach of \cite{abghr-infty}. Let $G$ be a space, $R$ be an $E_\infty$-ring spectrum, $\Pic(R)$ be the \emph{Picard space} of $R$ (the subspace of $\Mod_R$ spanned by the invertible $R$-modules), and $f:G\to \Pic(R)$ be a map. The colimit of $\xymatrix{G \ar[r]^-f & \Pic(R) \ar@^{^(-}[r] & \Mod_R}$ is the \emph{Thom $R$-module} of $f$, denoted $Mf$. In fact, $\Pic(R)$ is an $E_\infty$-group. When $G$ is an $E_\infty$-group and $f$ is a map of $E_\infty$-groups, then $Mf$ gets the structure of an $E_\infty$-$R$-algebra \cite{abg}, \cite{barthel-antolin}. Examples of Thom $E_\infty$-ring spectra include complex cobordism $MU$ and periodic complex cobordism $MUP$.

\begin{thmn}[\ref{thm-taqthom}]Let $R$ be an $E_\infty$-ring spectrum. Let $f:G\to \Pic(R)$ be a map of $E_\infty$-groups. There is an equivalence of $Mf$-modules
\[L_{Mf/R} \simeq Mf \sma B^\infty G.\]
\end{thmn}
Here $B^\infty G$ denotes the connective spectrum associated to $G$.

A model-categorical version had first appeared in \cite{basterra-mandell}. For example, we recover the equivalence of $MU$-modules \[L_{MU}\simeq MU \sma bu\] of that paper. The result of Basterra and Mandell, however, only applies to Thom spectra of maps to $BGL_1(\bS)$, whereas ours applies to maps to $\Pic(R)$ where $R$ is any $E_\infty$-ring spectrum. This allows for generalized, possibly non-connective Thom $E_\infty$-ring spectra. For example, we get that \[L_{MUP}\simeq MUP\sma ku\]%
as $MUP$-modules, where $ku$ denotes the $E_\infty$-ring spectrum of connective complex topological $K$-theory. In fact, $L_{MUP}$ is actually a Thom $E_\infty$-$ku$-algebra. More generally, we observe in \cref{prop-ringspace} that when $G$ is an \emph{$E_\infty$-ring space}, then the $R$-module $L_{Mf/R}$ underlies a Thom $E_\infty$-($R\sma B^\infty G$)-algebra. In other words, with this additional hypothesis the cotangent complex of a Thom $E_\infty$-algebra becomes a Thom $E_\infty$-algebra.

In \cref{section-etale} we extend \cref{thm-taqthom} to cotangent complexes of two types of extensions of Thom algebras: if $Mf\to B$ is a map $E_\infty$-$R$-algebras which is either étale or solid (i.e. the multiplication $B\sma_{Mf}B \to B$ is an equivalence), then
\[L_{B/R} \simeq B\sma B^\infty G.\]
This allows us, for example, to recover the equivalence $L_{KU}\simeq KU\sma H\Q$ from \cite{stonek-thhku}, where $KU$ denotes the $E_\infty$-ring spectrum of periodic complex topological $K$-theory.

\subsection{Notation and conventions}  \label{sect-notation}
We will freely use the language of $\infty$-categories as developed in \cite{htt}, \cite{ha}.

Let $\C$ be an $\infty$-category. Given a fixed map $f:A \to B$, we denote by $\C_\AB$ the $\infty$-category 
$(\C_{/B})_{f/}$ of objects $C\in \C$ together with maps $A\to C\to B$ which compose to $f$.
Similarly, we denote by $\C_\BB$ the $\infty$-category $(\C_{/B})_{\id_B/}$. 

If $\C$ is an $\infty$-category with terminal object $T$, we let $\C_*$ denote the undercategory $\C_{T/}$. 

The $\infty$-category of spaces %
will be denoted by $\S$, and that of spectra by $\Sp$. Its full subcategory of connective spectra is denoted by $\Spc$. We denote by $\CAlg$ the $\infty$-category $\CAlg(\Sp)$ of  $E_\infty$-ring spectra, and by $\CAlg_R$ that of $E_\infty$-algebras over an $E_\infty$-ring spectrum $R$, i.e. $\CAlg(\Mod_R)$. The suspension spectrum functor $\S\to \Sp$ is denoted $\Sigma^\infty_+$, and if $G\in \CAlg(\S)$, then $\bS[G]$ denotes the $E_\infty$-ring spectrum $\Sigma^\infty_+(G)$.

\subsection{Acknowledgments}
The authors would like to thank Hongyi Chu, Yonatan Harpaz and Piotr Pstr\k{a}gowski for helpful discussions, and Jacob Lurie for helping us find a better proof of \cref{Proposition} than the one we had, by suggesting the use of \cite[1.4.2.24]{ha}. 
The first-named author would like to thank the Institut des Hautes {\'E}tudes Scientifiques (IHES) for its hospitality and financial support.
The second-named author would like to thank the Max Planck Institute for Mathematics for its hospitality and financial support.

\section{Short review of stabilization and Goodwillie calculus}
\label{section-good}
 
Let us summarize some notions from the Goodwillie calculus of functors which we shall be using. We shall work with the $\infty$-categorical version of it as in \cite[Chapter 6]{ha}. For simplicity, let us assume that $\C$ and $\D$ are $\infty$-categories which are pointed and presentable.%

\begin{enumerate}
\item A functor $F:\C\to \D$ is \emph{reduced} if it takes a final object to a final object. It is \emph{excisive} if it takes pushout squares to pullback squares. The full subcategory of $\Fun(\C,\D)$ spanned by the excisive functors is denoted $\Exc(\C,\D)$, and the one spanned by the reduced, excisive functors is denoted $\Exc_*(\C,\D)$.
 
\item \label{item-spectra}The $\infty$-category of spectra in $\C$ is defined by $\Sp(\C) =\Exc_*(\S_*^\fin,\C)$ where $\S_*^\fin$ is the $\infty$-category of pointed finite spaces \cite[1.4.2.8]{ha}. Evaluation at the sphere $S^0$ defines a functor $\Omega^\infty:\Sp(\C)\to \C$ which is an equivalence when $\C$ is stable \cite[1.4.2.21]{ha}; 
since $\C$ is pointed and presentable, $\Omega^\infty$ admits a left adjoint $\Sigma^\infty:\C\to \Sp(\C)$ \cite[1.4.4.4]{ha}. If $\C$ is presentable but not pointed, the left adjoint to $\Omega^\infty$ is denoted $\Sigma^\infty_+$ and it factors into two left adjoint functors $\C\xrightarrow{(-)_+}  \C_* \xrightarrow{\Sigma^\infty} \Sp(\C_*)\simeq \Sp(\C)$ %
\cite[4.10]{ggn}.

\item \label{item-colim} The inclusion $\Exc(\C,\D)\to \Fun(\C,\D)$ has a left adjoint $P_1$, called the \emph{excisive approximation} functor \cite[6.1.1.10]{ha}. The unit natural transformation $F\Rightarrow P_1F$ is said to \emph{exhibit $P_1F$ as the excisive approximation to $F$}. If $F:\C\to \D$ is reduced, then $P_1F$ is reduced, %
and by \cite[6.1.1.28]{ha}, \[P_1F \simeq \colim_n (\Omega^n_\D \circ F \circ \Sigma^n_\C).\] 

\item \label{item-partialf} \cite[6.2.1.4]{ha} If $F:\C\to \D$ is left exact (i.e. preserves finite limits), then composition with $F$ defines a functor $\partial F:\Sp(\C)\to \Sp(\D)$, the \emph{(Goodwillie) derivative} which makes the following diagram commute
\[
\xymatrix{
\Sp(\C) \ar[r]^-{\partial F} \ar[d]_-{\Omega^\infty_\C} & \Sp(\D) \ar[d]^-{\Omega^\infty_\D} \\ \C \ar[r]_-F & \D.}
\]

\item \label{item-general} In the case of an arbitrary functor $F:\C\to \D$, derivatives admit a general definition \cite[6.2.1.1]{ha}: they consist of a functor $\partial F:\Sp(\C)\to \Sp(\D)$ and a natural transformation $F\circ \Omega^\infty_\C\Rightarrow \Omega^\infty_\D \circ \partial F$ satisfying some properties. Derivatives are unique up to canonical equivalence \cite[6.2.1.2]{ha}, and there are very general existence results: in particular, under the conditions on $\C$ and $\D$ which we have imposed, every $F:\C\to \D$ admits a derivative \cite[6.2.1.9/6.2.3.13]{ha}.

\item \label{item-p1f} \cite[Page 1071]{ha} If $F:\C\to \D$ is reduced and preserves filtered colimits, then \[P_1F\simeq \Omega^\infty_\D \circ \partial F \circ \Sigma^\infty_\C.\]

\item \label{item-id} $\partial \id_\C\simeq \id_{\Sp(\C)}$ %
and $P_1\id_\C \simeq \Omega^\infty_\C \circ \Sigma^\infty_\C \simeq \colim_n(\Omega^n_\C \circ \Sigma^n_\C)$. %
In particular, if $\C$ is stable then $\id_\C \simeq \colim_n(\Omega^n_\C \circ \Sigma^n_\C)$. %

\item \label{item-p1f2} %
If $F:\C\to \D$ is reduced, left exact and preserves filtered colimits, then \[P_1F\simeq F \circ \Omega^\infty_\C \circ \Sigma^\infty_\C,\] as follows from (\ref{item-p1f}) and (\ref{item-partialf}).

\end{enumerate}

\section{The cotangent complex} \label{section-cotangent}

In this section, we introduce the cotangent complex of a map of $E_\infty$-ring spectra and we give different expressions for it: via the augmentation ideal, via a stabilization process, i.e. as a sequential colimit, and via indecomposables.

Before going to $E_\infty$-ring spectra, let us say a word on the general definition. The relative cotangent complex according to Lurie is a suspension spectrum, in the following sense:

\begin{defn} Let $\C$ be a presentable $\infty$-category and $f:A\to B$ in $\C$. Consider the suspension spectrum functor
\[\xymatrix{\C_\AB \ar[r]^-{\Sigma^\infty_+} & \Sp(\C_\BB).}\]
The \emph{relative cotangent complex} of $f$ is the image of $A\xrightarrow{f}B\xrightarrow{\id}B$ by this functor, and it is denoted $L_{B/A}$. If $A$ is an initial object of $\C$, then $L_{B/A}$ is also denoted $L_B$ and it is called the \emph{absolute cotangent complex} of $B$.
\end{defn}

\begin{rmk} \label{rmk-sip} Let us say a word about the $\Sigma^\infty_+$ functor above.  Since $\id:B\to B$ is the terminal object of $\C_{/B}$, then \[\C_\BB =(\C_{/B})_{\id_B/} \simeq (\C_{/B})_*.\]
Note as well that $\C_\AB=(\C_{/B})_{f/}\simeq (\C_{A/})_{/f}$.  On the other hand, by \cite[7.3.3.9]{ha}, we have $(\C_\AB)_* \simeq \C_\BB$. Therefore, $\Sigma^\infty_+$ factors as
\[\xymatrix{\C_\AB \ar[r]^-{-\sqcup_A B} & \C_\BB \ar[r]^-{\Sigma^\infty} & \Sp(\C_\BB).}\]
\end{rmk}

In order to address the issue of functoriality of the cotangent complex, Lurie uses the \emph{tangent bundle} of $\C_\AB$. We shall not be needing this, so for the sake of simplicity we will not introduce it.

Let us now concentrate on the case $\C=\CAlg$.

\subsection{The cotangent complex via the augmentation ideal}

When $\C=\CAlg$, we may identify $\Sp(\C_\BB)$ with a more familiar $\infty$-category, namely $\Mod_B$, as we shall now see. Note that $\CAlg_\BB$ is the $\infty$-category of \emph{augmented} $E_\infty$-$B$-algebras: its objects are $E_\infty$-$B$-algebras $C$ with a map $C\to B$ of $E_\infty$-$B$-algebras.

\begin{defn} %
\label{def-i} Let $B\in \CAlg$. The \emph{augmentation ideal} functor
\[I:\CAlg_\BB\to \Mod_B\]
takes $C$ to the fiber in $\Mod_B$ of the augmentation, i.e. to the pullback
\[\begin{tikzcd} [row sep=.6cm, column sep=.6cm]
I(C) \arrow[r] \arrow[d] \arrow[dr, phantom, "\lrcorner", very near start]
& C \arrow[d] \\
0 \arrow[r] 
& B
\end{tikzcd}
\]
in $\Mod_B$, where $0$ denotes a zero object of $\Mod_B$.\end{defn}

\begin{rmk}The functor $I$ is right adjoint to the functor $\Mod_B \to \CAlg_\BB$ which takes a module $M$ to the free $E_\infty$-$B$-algebra $\bigvee_{n\geq 0} (M^{\sma_B n})_{\Sigma_n}\in \CAlg_B$ endowed with the augmentation over $B$ given by projection to the $0$-th summand. Here $(-)_{\Sigma_n}$ denotes the (homotopy) orbits for the $\Sigma_n$-action \cite[3.1.3.14, 7.3.4.5]{ha}.
Note that $I$ is reduced, as it takes $B$ to the zero module.
\end{rmk}

\begin{notn} We let $\nuca_B$ denote the $\infty$-category of non-unital $E_\infty$-$B$-algebras, which we call \emph{nucas}\footnote{The \emph{c} in \emph{nuca} stands for \emph{commutative}.} \cite[5.4.4.1]{ha}. %
\end{notn}

\begin{rmk}\label{rmk-nuca}The augmentation ideal functor factors through $\nuca_B$ as follows:
\[\xymatrix@C-.8pc@R-.8pc{\CAlg_\BB \ar[rr]^-I \ar[rd]_-{I_0} && \Mod_B \\ & \nuca_B \ar[ru]_-U}\]
where $U$ is the forgetful functor. Indeed, one can take the pullback defining $I$ in \cref{def-i} in the $\infty$-category $\nuca_B$ instead of in $\Mod_B$, which defines $I_0$. 
The functor $I_0$ is a right adjoint to the functor $N\mapsto B\vee N$, and it is in fact an equivalence \cite[5.4.4.10]{ha}, \cite[2.2]{basterra}. Note as well that $I$ commutes with sifted colimits, since $U$ does \cite[3.2.3.1]{ha}. %
In particular, $I$ commutes with filtered or sequential colimits.
\end{rmk}

\begin{thm} \cite[7.3.4.7/14]{ha} \label{thm-spcalg} The functor
\[\xymatrix{\Sp(\CAlg_\BB) \ar[r]^-{\partial I} & \Sp(\Mod_B)}\]
is an equivalence of $\infty$-categories; in particular, $\Sp(\CAlg_\BB) \xrightarrow{\partial I} \Sp(\Mod_B) \xrightarrow{\Omega^\infty} \Mod_B$ is an equivalence as well.
\end{thm}

\begin{rmk} %
\label{rmk-bm-model} A model-categorical precedent can be found as Theorem 3 of \cite{basterra-mandell}. There, the functor fitting in the place of $\partial I$ is defined as follows. First of all, in their framework a spectrum in a model category $\M$ is a sequence of objects $\{X_n\}_{n\geq 0}$ of $\M$ with maps $\Sigma X_n\to X_{n+1}$. Spectra in $\M$ have a model structure whose fibrant objects are the $\Omega$-spectra. Thus, any topological left Quillen functor $F$ between model categories enriched over based spaces induces a left Quillen functor $\underline{F}$ between the corresponding model categories of spectra: the arrows $\Sigma X_n \to X_{n+1}$ get sent to $\Sigma F(X_n)\simeq F(\Sigma X_n) \to F(X_{n+1})$. 

In particular, the augmentation ideal functor from the model category of augmented commutative $B$-algebras to the model category of $B$-modules, let us also call it $I$, induces a functor $\underline{I}$ between the respective model categories of spectra. 
After passing to their underlying $\infty$-categories, $\underline{I}$ gives a functor equivalent to $\partial I$. %
\end{rmk}

\begin{rmk}  \label{rmk-commuting}
Recall from \ref{section-good}.(\ref{item-partialf}) that $\Omega^\infty_{\Mod_B} \circ \partial I \simeq I\circ \Omega^\infty_{\CAlg_\BB}$, i.e. $\partial I$ commutes with $\Omega^\infty$. On the other hand, note that $\partial I$ typically does not commute with $\Sigma^\infty$. If it did, then \[\Omega^\infty_{\Mod_B} \circ \partial I \circ \Sigma^\infty_{\CAlg_\BB} \simeq \Omega^\infty_{\Mod_B} \circ \Sigma^\infty_{\Mod_B} \circ I \simeq I,\] but on the other hand this is equivalent to $I \circ \Omega^\infty_{\CAlg_\BB} \circ \Sigma^\infty_{\CAlg_\BB}$ which is the excisive approximation of $I$ by \ref{section-good}.(\ref{item-p1f2}). Therefore, $I$ would be excisive. %
Since $\Mod_B$ is stable, this would mean that $I$ preserves pushouts; since $I$ is also reduced, then $I$ would be right exact. But this is typically false. %
For example, $I$ does not commute with coproducts: if we take $B=\bS$, the coproduct of $\bS[S^1]$ with itself in $\CAlg_{\under{\bS}{\bS}}$ is $\bS[S^1\times S^1]$. Its augmentation ideal is $\Sigma^\infty (S^1\times S^1)\simeq \Sigma^\infty S^1\vee\Sigma^\infty S^1 \vee \Sigma^\infty S^2$, which is not equivalent to the coproduct of $I(\bS[S^1])\simeq \Sigma^\infty S^1$ with itself. %
\end{rmk}

If $f:A\to B\in \CAlg$, then $L_{B/A}\in \Sp(\CAlg_\BB)$ by definition. Given \cref{thm-spcalg}, in this situation we redefine $L_{B/A}$ to mean the image of $A\xrightarrow{f}B\xrightarrow{\id}B$ under the composition
\begin{equation}\label{eq-deftaq}
\xymatrix@C+1pc{\CAlg_\AB \ar[r]^-{B\sma_A-} & \CAlg_\BB \ar[r]^-{\Sigma^\infty} & \Sp(\CAlg_\BB) \ar[r]^-{\Omega^\infty \circ \partial I}_-\simeq & \Mod_B.}
\end{equation}
Therefore, by \ref{section-good}.(\ref{item-p1f}), $L_{B/A}$ is equivalently the value of an excisive approximation to $I:\CAlg_\BB\to \Mod_B$ evaluated in $B\sma_A B$. In symbols,
\begin{equation}\label{eq-p1} L_{B/A}\simeq (P_1I)(B\sma_A B).\end{equation}

\subsection{The cotangent complex as a colimit} 
Let $A$ be an $E_\infty$-ring spectrum. The general definition of a cotangent complex also applies to an $E_k$-$A$-algebra $B$. In \cite[7.3.5]{ha} Lurie analyzes this particular case. He denotes by $L^{(k)}_{B/A}$ the resulting \emph{$E_k$-cotangent complex}.

Forgetting structure, every $E_\infty$-$A$-algebra $B$ is an $E_k$-$A$-algebra for every $k\geq 0$. Lurie observes in \cite[7.3.5.6]{ha} that since the $E_\infty$-operad is the colimit of the $E_k$-operads, these $E_k$-cotangent complexes recover the 
cotangent complex as follows:
\[L_{B/A} \simeq  \colim (\xymatrix{L_{B/A}^{(1)} \ar[r] &  L_{B/A}^{(2)} \ar[r] & L_{B/A}^{(3)} \ar[r] & \cdots }).\]
These $E_k$-cotangent complexes admit a different expression which is sometimes computable, as we shall see in this section. That is what we shall use in \cref{section-thom} to compute the cotangent complex of Thom $E_\infty$-algebras.\\

Let $B\in \CAlg$ %
and $C\in \CAlg_\BB$. Since $I$ is left exact (it is a right adjoint) and commutes with sequential colimits (\cref{rmk-nuca}), then by \ref{section-good}.(\ref{item-colim}),
\begin{equation}
\label{eq-colim}
(P_1I)(C) \simeq \colim (\xymatrix@C+1pc{I(C) \ar[r]^-{e_0} & \Omega I(\Sigma C) \ar[r]^-{\Omega e_1} & \Omega^2 I(\Sigma^2C) \ar[r]^-{\Omega^2e_2} & \cdots }).
\end{equation}
Here $e_n:I(\Sigma^nC) \to \Omega I(\Sigma^{n+1} C)$ is the natural map obtained as in \cite[1.4.2.12]{ha}. Explicitly, it is obtained as follows. Write $\Sigma^{n+1}C$ as the pushout of $B\leftarrow \Sigma^n C \to B$ (remember that $B$ is a zero object of $\CAlg_\BB$). %
Apply $I$, then $e_n$ is defined as the universal pullback map $ I(\Sigma^n C)\to \Omega I(\Sigma^{n+1}C)$.

Let $f:A\to B$ be a morphism in $\CAlg$. Applying (\ref{eq-colim}) to $C=B\sma_A B$ we get a quite explicit colimit formula for $L_{B/A}$. But we can be more explicit: we are going to recast  $I(\Sigma^n (B\sma_A B))$ in other terms. \\

Any presentable $\infty$-category $\C$ is tensored over spaces: there is a functor $-\otimes -:\S \times \C \to \C$ which preserves colimits separately in each variable. If $X\in \S$ and $c\in \C$, then \[X\otimes c\simeq \colim(X\xrightarrow{\{c\}} \C)\] where $\{c\}$ denotes the constant functor with value $c$. If $\C$ is moreover pointed, then it is tensored over pointed spaces: there is a functor $-\odot-:\S_*\times \C\to \C$ which preserves colimits separately in each variable. If $(X,x_0)\in \S_*$ and $c\in \C$, then
\begin{equation}\label{eq-odot}X\odot c \simeq \cof_\C(c\simeq *\otimes c \stackrel{x_0 \otimes \id}{\longrightarrow} X\otimes c).
\end{equation}
See \cite[Section 2]{rsv-thom} for more details.

\begin{rmk} \label{sn-odot}The suspension $\Sigma A$ of an object $A$ in a pointed presentable $\infty$-category $\C$ can be expressed as $S^1\odot A$. Indeed, write $S^1=\colim (*\leftarrow S^0 \to *)$ and apply the colimit-preserving functor $-\odot A$. By induction, $\Sigma^nA \simeq S^n\odot A$ for all $n\geq 0$.
\end{rmk}

\begin{notn} Let $\odot_B$ denote the tensor of $\CAlg_\BB$ over $\S_*$. Let $\otimes_A$ denote the tensor of $\CAlg_A$ over $\S$.
\end{notn}

\begin{rmk} \label{rmk-uniaug} If $f:A\to B$ in $\CAlg$ and $(X,x_0)\in \S_*$, then $X\otimes_A B\in \CAlg_\BB$, with unit and augmentation given by \[B\simeq *\otimes_A B \xrightarrow{x_0\otimes \id} X\otimes_A B \xrightarrow{*\otimes \id} *\otimes_A B \simeq B.\]
More generally, if $(B\xrightarrow{g} C\xrightarrow{e} B) \in \CAlg_\BB$, then $X\otimes_A C\in \CAlg_\BB$, with unit and augmentation given by
\[B\simeq *\otimes_A B \xrightarrow{x_0\otimes g} X\otimes_A C \xrightarrow{*\otimes e} *\otimes_A B \simeq B.\]
\end{rmk}

We shall now prove a couple of results about this construction.

\begin{rmk} \label{rmk-cis}
The definition of an adjunction in \cite[5.2]{htt}, which we are implicitly adopting, uses the theory of correspondences. We shall use the result of Cisinski \cite[6.1.23; Footnote, Page 250]{cisinski} which says that Lurie's definition is equivalent to the expected characterization via natural equivalences of mapping spaces $\Map_\C(c,Gd) \simeq \Map_\D(Fc,d)$.
\end{rmk}

In the following lemma, we shall need the following notation: if $F:\C\to \D$ is a functor of $\infty$-categories and $c\in \C$, we denote by $\overline{F}:\C_{c/}\to \D_{F(c)/}$ the induced functor on undercategories. Note that if $F$ has a right adjoint $G$, then by \cref{rmk-cis} we conclude that $\overline{F}$ also has a right adjoint. Indeed, if we are given maps $c\to c'$ and $c\to Gd$ in $\C$, then the natural equivalence
\[\Map_\C(c',Gd) \simeq \Map_\D(Fc',d)\]
restricts to a natural equivalence
\[\Map_{\C_{c/}}(c',Gd) \simeq \Map_{\D_{Fc/}}(Fc',d).\]
Explicitly, the right adjoint of $\overline{F}$ takes an object $g:Fc \to d$ to the composition $c \to GFc \xrightarrow{Gg} Gd$.  Therefore, if the $\infty$-categories are presentable and $F$ preserves colimits, then $\overline{F}$ also preserves colimits.  %

We will now consider bifunctors that preserve colimits separately in each variable and analyze in which way this property passes on to undercategories.

\begin{lemma} \label{lemma-bifunctor}
 Let $F: \C \times \D \to \E$ be a functor of presentable $\infty$-categories that preserves colimits separately in each variable. 
 Let $(c,d)\in\C \times \D$. Consider the induced functor
\[\overline{F}:\C_{c/}\times \D_{d/}\to \E_{F(c,d)/}.\]
For a given $h:c \to c'$ in $\C$, the functor
$$\overline{F}(h, - ) : \D_{d/} \to \E_{F(c,d)/}$$ 
preserves colimits if and only if $F(h,\id_d): F(c,d) \to F(c',d)$ is an equivalence.
\end{lemma}

There is an analogous result in the other variable, starting from an arrow $d\to d'\in \D$, but we shall not be needing it.

\begin{proof}
We can factor the functor $\overline{F}(h, - )$ as the composition
$$ \D_{d/} \xrightarrow{ \overline{F(c',-)} } \E_{F(c',d)/} \xrightarrow{F(h,\id_d)_* } \E_{F(c,d)/}.$$
The first functor preserves colimits %
by the discussion above applied to $F(c',-): \D \to \E$, which preserves colimits by hypothesis. Therefore $\overline{F}(h,-)$ preserves colimits if and only if $F(h,\id_d)_*$ preserves colimits. 

If $F(h,\id_d)_*$ preserves colimits, then it preserves initial objects, which forces $F(h,\id_d)$ to be an equivalence. The converse is obvious.
\end{proof}

\begin{rmk} The functor $\overline{F}$ of \cref{lemma-bifunctor} always preserves colimits indexed by weakly contractible simplicial sets separately in each variable, by \cite[4.4.2.9]{htt}. We shall not be using this fact, though.
\end{rmk}

\begin{prop}  \label{prop-extendtensor}
\begin{enumerate}
\item The functor $- \otimes_A -: \S \times \CAlg_A \to \CAlg_A$
extends to a functor 
\begin{equation}\label{eq-funer}
- \otimes_A -: \S_* \times \CAlg_\BB \to \CAlg_\BB
\end{equation}
in the sense that the following diagram commutes, where the two vertical maps are forgetful functors:
\[\xymatrix{
\S_* \times \CAlg_\BB \ar[r]^-{-\otimes_A -} \ar[d] & \CAlg_\BB \ar[d] \\ \S \times \CAlg_A \ar[r]_-{-\otimes_A-} & \CAlg_A.}
\]
The functor (\ref{eq-funer}) takes $(X,C)$ to $X\otimes_A C$ with unit and augmentation given as in \cref{rmk-uniaug}.
 
\item  %
Letting the second variable of (\ref{eq-funer}) be of fixed value  $(B \to C \to B) \in \CAlg_\BB$, 
 the restricted functor%
 $$- \otimes_A C: \S_* \to \CAlg_\BB$$ 
 preserves colimits if and only if the unit %
 $B \to C$ %
 is an equivalence.

\item Letting the second variable of (\ref{eq-funer}) be of fixed value $B \xrightarrow{\id}B \xrightarrow{\id} B$, the restricted functor \[-\otimes_A B: \S_*\to \CAlg_\BB\] is equivalent to $-\odot_B (B\sma_A B)$, where $B\sma_A B$ denotes the object $B \xrightarrow{\iota_0} B\sma_AB \xrightarrow{\mu} B$ of $\CAlg_\BB$.\footnote{The inclusion $\iota_0$ can be replaced by $\iota_1$: up to composing with the symmetry $B\sma_A B\xrightarrow{\sim} B\sma_A B$, the two choices are equivalent.}
 
 \end{enumerate}
\end{prop}

\begin{proof}
\begin{enumerate}
\item
 Since the tensor is a colimit and colimits in overcategories are created in the original $\infty$-category \cite[1.2.13.8]{htt}, the functor $-\otimes_A-:\S\times \CAlg_A\to \CAlg_A$ begets a functor 
\begin{equation}
\label{eq-tensim} -\otimes_A-:\S\times (\CAlg_A)_{/B} \to (\CAlg_A)_{/B}\end{equation}
which extends the original one and preserves colimits separately in each variable.

Now, note as in \cref{rmk-sip} that
\[(\CAlg_A)_{/B} \simeq (\CAlg_{A/})_{/f} \simeq \CAlg_\AB.\]
We now consider $(*,A\to B\to B)\in \S_*\times \CAlg_\AB$ and consider the functor induced by (\ref{eq-tensim}) in undercategories. This gives the result, since $A\to B\to B$ is a terminal object of $\CAlg_\AB$, and by \cref{rmk-sip} $(\CAlg_\AB)_*\simeq \CAlg_\BB$.

The functor we just obtained acts on objects as expected, by construction.

  \item 
This follows from %
\cref{lemma-bifunctor}, since the condition there amounts in this case to the map $* \otimes_A B \to * \otimes_A C$ being an equivalence.
  \item
To prove $-\otimes_A B$ and $-\odot_B (B\sma_A B)$ are equivalent, it suffices to see that they send $S^0$ to equivalent objects. Indeed, they are colimit-preserving functors from $\S_*$ to a pointed presentable $\infty$-category, but $\S_*$ is freely generated under colimits by $S^0$ in pointed presentable $\infty$-categories \cite[2.29]{rsv-thom}.

Now, indeed both functors send $S^0$ to $B\sma_A B$, and the proof is finished.\qedhere 
   \end{enumerate}
\end{proof}

\begin{lemma}\label{lemma-ij} Let $B \xrightarrow{ \ u \ } C \xrightarrow{ \ c \ } B$ be an object in $\CAlg_\BB$. 
Then $I(C)$ is naturally equivalent to the cofiber of $u$ in $\Mod_B$, i.e.
\[I(C)=\fib_{\Mod_B}(c:C\to B) \simeq \cofib_{\Mod_B}(u:B\to C) \eqqcolon J(C).\]
\begin{proof}Consider the following commutative diagram in $\Mod_B$:
\[\xymatrix@C-.2cm@R-.2cm{0 \ar@/^.6cm/[rr]^-\id \ar[r] \ar[d] & B \ar@/^.7cm/[dd]|\hole^(.7)\id \ar[d]_-u \ar[r] & 0 \ar[d] \\ I(C) \ar[r] \ar[d] & C \ar[r] \ar[d]_-c  & J(C) \\ 0 \ar[r] & B.}\]
The bottom left square and the big rectangle formed by the two squares on the left are pullbacks, so the square on the top left is a pullback \cite[4.4.2.1]{htt}. It is therefore a pushout, by stability of $\Mod_B$. The square on the top right is also a pushout, hence the big rectangle formed by the two squares on top is a pushout \cite[Dual of 4.4.2.1]{htt}, proving the result. 
\end{proof}
\end{lemma}

We adopt the notation of \cite{kriz} or \cite[Page 164]{schlichtkrull-higher}:

\begin{notn} \label{not-otilde} Let $X\in \S_*$ and $f:A\to B$ in $\CAlg$. By the previous lemma, there is an equivalence $I(X\otimes_A B)\simeq J(X\otimes_A B)$ in $\Mod_B$ natural in $X$ and $B$: we denote the common value by $X\otilde_A B$.
\end{notn}

From \cref{prop-extendtensor} and \cref{lemma-ij}, we deduce:

\begin{cor} \label{cor-xodot} Let $f:A\to B$ in $\CAlg$ and $X\in \S_*$. There is an equivalence of $B$-modules
\[I(X\odot_B (B\sma_A B)) \simeq X\otilde_A B\]
natural in $X$.
\end{cor}

We can now recast (\ref{eq-colim}) in a different form. 
Define arrows $e_n': S^n\otilde_A B \to \Omega (S^{n+1}\otilde_A B)$ as follows. Start by presenting $S^{n+1}\in \S_*$ as the pushout of $*\leftarrow S^n\to *$. Apply $-\otimes_A B:\S_*\to \CAlg_\BB$ to it, and then apply $I$. Now $e_n'$ is the universal pullback map $S^n\otilde_A B \to \Omega (S^{n+1}\otilde_A B)$.

\begin{prop} \label{prop-colim} Let $f:A\to B$ in $\CAlg$. There is an equivalence of $B$-modules 
\[L_{B/A}\simeq \colim_{\Mod_B} (\xymatrix@C+1pc{S^0\otilde_A B \ar[r]^-{e_0'} & \Omega (S^1\otilde_A B) \ar[r]^-{\Omega e_1'} & \Omega^2 (S^2\otilde_A B \ar[r]^-{\Omega^2 e_2'}) & \cdots }) \]
where $\Omega$ denotes the loop functor in $\Mod_B$.
\begin{proof} We use the characterization $L_{B/A}\simeq (P_1I)(B\sma_A B)$ from (\ref{eq-p1}). Consider the equivalence (\ref{eq-colim}) with $C=B\sma_A B$. By \cref{sn-odot} and \cref{cor-xodot}, we have
\[I(\Sigma^n(B\sma_A B)) \simeq I(S^n\odot_B (B\sma_A B)) \simeq S^n\otilde_A B.\]
The maps $e_n$ from (\ref{eq-colim}) and the maps $e_n'$ are defined in an analogous fashion, so the natural equivalences above commute with them.
\end{proof}
\end{prop}

\begin{rmk} The previous proposition was known to the experts (it is mentioned e.g in \cite[Page 164]{schlichtkrull-higher}), but we do not think a complete derivation had been spelled out in the literature before.
\end{rmk}

Finally, let us make the connection between the $\otilde$ construction and the $E_k$-cotangent complex mentioned at the beginning of the subsection.

\begin{rmk} Let $B$ be an $E_\infty$-$A$-algebra. Forgetting structure, we may consider $B$ as an $E_k$-$A$-algebra, for all $k\geq 0$, %
and thus we may form its $E_k$-cotangent complex $L^{(k)}_{B/A}$ \cite[7.3.5]{ha}.  
Lurie \cite[7.3.5.1/3]{ha} proved that for each $k\geq 1$ there is a fiber sequence of $B$-modules
\[\xymatrix{
L_{B/A}^{(k)} \ar[r] & \Omega^{k-1}(S^{k-1} \otimes_A B) \ar[rr]^-{\Omega^{k-1}(*\otimes \id)} && \Omega^{k-1} B,} \]
where $*:S^{k-1}\to *$ denotes the unique map. Since $*\otimes\id:S^{k-1}\otimes_A B\to B$ is the augmentation of $S^{k-1}\otimes_A B$ and loops commute with pullbacks, this identifies $L_{B/A}^{(k)}$ with $\Omega^{k-1}(S^{k-1}\otilde_A B)$ (see also \cite[1.3]{basterra-mandell-En}), so by \cref{prop-colim} we obtain an equivalence of $B$-modules
\[L_{B/A} \simeq  \colim (\xymatrix{L_{B/A}^{(1)} \ar[r] &  L_{B/A}^{(2)} \ar[r] & L_{B/A}^{(3)} \ar[r] & \cdots })\]
recovering \cite[7.3.5.6]{ha}.
\end{rmk}

\subsection{The cotangent complex via indecomposables}

The first published definition of the cotangent complex was established in the context of model categories using the indecomposables functor \cite{basterra}.
The goal of this subsection is to prove that this definition of the cotangent complex is equivalent to the definition adopted in (\ref{eq-deftaq}). We are not aware of a discussion of the approach using indecomposables in the $\infty$-categorical setting.

The content of this subsection will not be used in the sequel: the reader interested in Thom spectra should feel free to jump ahead to \cref{section-thom}.

\begin{defn} Let $B\in \CAlg$. We denote by \[Q:\nuca_B\to \Mod_B\] the \emph{indecomposables} functor that takes $N$ to the cofiber in $\Mod_B$ of the multiplication, i.e. to the pushout
\[\begin{tikzcd} [row sep=.6cm, column sep=.6cm]
N \sma_B N \arrow[r,"\mu"] \arrow[d] \arrow[dr, phantom, "\ulcorner", very near end]
& N \arrow[d] \\
0 \arrow[r] 
& Q(N)
\end{tikzcd}
\]
in $\Mod_B$. %
\end{defn}

The functor $Q$ is left adjoint to the functor which takes a module $M$ and endows it with a zero multiplication map. More precisely:

\begin{lemma} \label{lemma-Q}\begin{enumerate}
\item The functor $Q$ is left adjoint to a functor $Z:\Mod_B\to \nuca_B$ such that $UZ\simeq \id_{\Mod_B}$, and for each $M\in \Mod_B$ the multiplication map $ZM\sma_B ZM \to ZM$ is zero. 
\item \label{lemma-Q-id} $Q\circ F\simeq \id_{\Mod_B}$, where $F:\Mod_B\to \nuca_B$ is the free functor.
\item There exists a unique functor $Z:\Mod_B\to \nuca_B$ such that $UZ\simeq \id_{\Mod_B}$, up to equivalence.
\end{enumerate}
\end{lemma}
Here $U:\nuca_B\to \Mod_B$ denotes the forgetful functor.

\begin{proof}
\begin{enumerate}
\item We will use a criterion for adjointness from %
\cite[5.2.4.2]{htt}: $Q$ admits a right adjoint if and only if for every $M\in \Mod_B$, the comma $\infty$-category $(Q\downarrow M)$, defined as the pullback
\[\begin{tikzcd} [row sep=.6cm, column sep=.6cm]
(Q\downarrow M) \arrow[r] \arrow[d] \arrow[dr, phantom, "\lrcorner", very near start]
& \nuca_B \arrow[d] \\
(\Mod_B)_{/M} \arrow[r] 
& \Mod_B
\end{tikzcd}\]
has a terminal object. Fix $M\in \Mod_B$. It suffices to see that there exists a $B$-nuca $ZM$ such that $M\simeq QZM$, for then $QZM\simeq M\xrightarrow{\id} M$ is a terminal object of $(Q\downarrow M)$.

To see this, first recall from \cref{rmk-nuca} that the augmentation ideal functor $I:\CAlg_\BB \to \Mod_B$ factors via an equivalence $I_0:\CAlg_\BB \to \nuca_B$ followed by the forgetful functor. Now consider the trivial square-zero extension $B\oplus M$ \cite[7.3.4.16]{ha}. This is an augmented $E_\infty$-$B$-algebra such that the multiplication map of $I_0(B\oplus M)$ is zero. Define $ZM$ to be $I_0(B\oplus M)$, so clearly $UZM\simeq M$. By definition, $QZM$ is the cofiber in $\Mod_B$ of the multiplication map $ZM \sma_B ZM\to ZM$ which is zero, so $QZM\simeq M$ as required.

\item $Q\circ F$ is the left adjoint to $U\circ Z\simeq \id_{\Mod_B}$, so it is equivalent to the identity.
\item %
The existence of such a $Z$ has just been proven. Now suppose we have a functor $Z':\Mod_B\to \nuca_B$ such that $UZ'\simeq \id_{\Mod_B}$. %
Let $M,M'\in \Mod_B$. We have natural equivalences of spaces
\[\Map_{\nuca_B}(FM',Z'M)\simeq \Map_{\Mod_B}(M',UZ'M) \simeq \Map_{\Mod_B}(QFM',M).\]

Let $N\in \nuca_B$. Note that the free-forgetful adjunction $(F,U)$ is monadic by %
\cite[4.7.3.5]{ha}, since $\nuca_B$ is an $\infty$-category of algebras over an $\infty$-operad in $\Mod_B$ \cite[5.4.4.1]{ha} so the forgetful functor is conservative and preserves geometric realizations of simplicial objects \cite[3.2.2.6/3.2.3.2]{ha}. %
Therefore, by \cite[4.7.3.14/15]{ha}, there exists a simplicial object $N_\bullet$ in $\nuca_B$ which depends functorially on $N$ and satisfies that $\colim(N_\bullet)\simeq N$, $N_\bullet$ is given by free nucas; moreover, $\colim(QN_\bullet)\simeq QN$ by \cite[4.7.2.4]{ha}. Using the above, one obtains a natural equivalence of spaces
\[\Map_{\nuca_B}(N,Z'M)\simeq \Map_{\Mod_B}(QN,M).\]
By \cref{rmk-cis}, this proves that $Z'$ is a right adjoint to $Q$, 
but then by uniqueness of adjoints \cite[5.2.6.2]{htt},  %
we deduce that $Z'\simeq Z$. \qedhere %
\end{enumerate}
\end{proof}

\begin{rmk} Let $K$ denote the composition
\[\Mod_B\xrightarrow{\sim} \Sp(\CAlg_\BB)\xrightarrow{\Omega^\infty} \CAlg_\AB;\] 
here, the first arrow is an inverse to $\Omega^\infty \circ \partial I:\Sp(\CAlg_\BB) \xrightarrow{\sim}\Mod_B$. Note that $K(M)$ is the trivial square-zero extension $B\oplus M$ \cite[7.3.4.16]{ha}. By definition of $K$ and $L_{B/A}$, we get the following equivalence for every $B$-module $M$,%
\[\Map_{\Mod_B}(L_{B/A},M) \simeq \Map_{\CAlg_\AB}(B,B\oplus M).\]
These spaces can be interpreted as the spaces of $A$-linear \emph{derivations} from $B$ into $M$ \cite[7.3]{ha}.%

Now, observe that $K$ is equivalent to the composition \[\Mod_B \xrightarrow{Z} \nuca_B \xrightarrow{B\vee -} \CAlg_\BB\to \CAlg_\AB,\]
the last functor being a forgetful functor. Indeed, since $B\vee -$ is an equivalence with inverse given by $I_0$, we may equivalently verify that $I_0\circ K \simeq Z$. But this follows from Part (2) of the previous lemma. %
\end{rmk}

In the following theorem, we prove how to get the cotangent complex via indecomposables. A model categorical version of this result can be found in \cite[Theorem 4]{basterra-mandell}.

\begin{thm} \label{thm-three} Let $B\in \CAlg$. The following diagram commutes:
\[\xymatrix{
\CAlg_\BB \ar[r]^-{\Sigma^\infty} \ar[d]_-{I_0} & \Sp(\CAlg_\BB) \ar[d]_-\simeq^-{\Omega^\infty \partial I %
} \\ \nuca_B \ar[r]_-Q & \Mod_B.
}\]
\end{thm}

\begin{proof} Recall from \ref{section-good}.(\ref{item-p1f}) that $P_1I$, the excisive approximation to $I$, is equivalent to $\Omega^\infty \circ \partial I \circ \Sigma^\infty$. Thus, we have to prove that $Q\circ I_0$ is equivalent to $P_1I$. Recall that $I\simeq U \circ I_0$. By \cite[6.1.1.30]{ha}, we have $P_1I \simeq P_1U \circ I_0$. We will now prove that $P_1U \simeq Q$, finishing the proof.

By \ref{section-good}.(\ref{item-p1f}), $P_1U$ is equivalent to $\Omega^\infty_{\Mod_B} \circ \partial U \circ \Sigma^\infty_{\nuca_B}$. Note that $Q$ is excisive, since it preserves pushouts and $\Mod_B$ is stable, so $Q\simeq P_1Q$. Therefore, to prove that $P_1U\simeq Q$ it suffices to prove that $\partial U\simeq \partial Q$, by \ref{section-good}.(\ref{item-p1f}) once more.

We have the free--forgetful adjunction $\adj{\Mod_B}{\nuca_B}{F}{U}$, 
and taking derivatives gives an adjunction $\adj{\Sp(\Mod_B)}{\Sp(\nuca_B)}{\partial F}{\partial U}$ by \cite[6.2.2.15]{ha}. By \cref{lemma-Q}.(\ref{lemma-Q-id}), $Q\circ F \simeq \id_{\Mod_B}$, so by \cite[6.2.1.4/24]{ha} we get $\partial Q \circ \partial F \simeq \id_{\Sp(\Mod_B)}$. To prove that $\partial Q\simeq \partial U$, it suffices to prove that $\partial U$ is an equivalence. Indeed, if it is, then $\partial F \circ \partial U\simeq \id_{\Sp(\nuca_B)}$, so  $\partial Q \simeq \partial Q \circ \partial F \circ \partial U \simeq \partial U$.

To prove that $\partial U$ is an equivalence, we proceed similarly as in \cite[Proof of 7.3.4.7]{ha}: namely, since $U$ is a monadic functor, then by \cite[6.2.2.17]{ha} it suffices to prove that the unit $\id_{\Mod_B} \Rightarrow U\circ F$ induces an equivalence of derivatives. The proof of this is very similar to that of \cite[7.3.4.10]{ha}, only simpler.%

Note that $U\circ F:\Mod_B\to \Mod_B$ is given on objects by $M\mapsto M\vee \bigvee_{n\geq 2} (M^{\sma_B n})_{\Sigma_n}$ and the unit of the $(F,U)$ adjunction includes $M$ into the separate $M$ factor, so by \cite[7.3.4.8]{ha} it suffices to see that the derivative of the functor $\Sym^n:\Mod_B\to \Mod_B, M\mapsto (M^{\sma_B n})_{\Sigma_n}$ is nullhomotopic for $n\geq 2$.

To see this, note that $\Sym^n$ is the composition
\[\xymatrix{\Mod_B \ar[r]^-\diag & \Mod_B^n \ar[rr]^-{-\sma_B  \cdots  \sma_B-} && \Fun(B\Sigma_n,\Mod_B) \ar[r]^-\colim &\Mod_B}\]
where the first functor is the diagonal functor and the second functor takes $(M_1,\dots,M_n)$ to $M_1\sma_B \cdots \sma_B M_n$ together with the action by $\Sigma_n$ which permutes the factors. By \cite[7.3.4.8]{ha} the derivative operator $\partial:\Fun_*(\Mod_B,\Mod_B)\to \Exc(\Mod_B,\Mod_B)$ preserves colimits, %
so the derivative of $\Sym^n$ is the colimit of a functor $B\Sigma_n \to \Exc(\Mod_B,\Mod_B)$ with value $\partial \left( (-)^{\sma_B n}:\Mod_B\to \Mod_B\right)$. Therefore, it suffices to see that the functor $(-)^{\sma_B n}:\Mod_B\to \Mod_B$, $n\geq 2$ has nullhomotopic derivative.

By \cite[6.1.3.12]{ha}, this functor is $n$-reduced, %
so it is 2-reduced, which by definition means that its excisive approximation is nullhomotopic.\footnote{We have just made use of some higher Goodwillie calculus. The definition of an $n$-reduced functor can be found in \cite[6.1.2.1]{ha}, and the fact that any $n$-reduced functor is $(n-1)$-reduced follows from \cite[6.1.1.10/14]{ha}.} By \ref{section-good}.(\ref{item-p1f}), its derivative is nullhomotopic as well, since $\Mod_B$ is stable.
\end{proof}

\begin{rmk}
 The key aspect of the previous proof is the fact that $\partial (U \circ F) \simeq \id_{\Sp(\Mod_B)}$, which is proven in the last paragraph.
 Notice this is equivalent to
 \[P_1(U \circ F) \simeq \id_{\Mod_B},\]
 as $P_1(U \circ F) \simeq \Sigma^\infty \circ \partial (U \circ F) \circ \Omega^\infty$ by \ref{section-good}.(\ref{item-p1f}), 
 and $\Sigma^\infty, \Omega^\infty$ are equivalences since $\Mod_B$ is stable.
 Using the analogy of Goodwillie calculus with classical calculus, we can gain some intuition for this result. %
 Under this analogy, functors correspond to smooth functions of the real line, so the functor $U \circ F$ which maps $M\in \Mod_B$ to $\bigvee_{n\geq 1} (M^{\sma_B n})_{\Sigma_n}$ 
 corresponds to the power series $f(x) = \sum_{n=1}^\infty x^n$. %
The linear approximation at 0 of this function is the identity map $x\mapsto x$.
Continuing with the analogy, linear approximations of functions correspond to $1$-excisive approximations of functors, which 
 provides some intuition for the equivalence $P_1(U \circ F) \simeq \id_{\Mod_B}$. 
\end{rmk}
From \cref{thm-three} we immediately deduce:
\begin{cor} \label{cor-indecomposables}
Let $f:A\to B$ be a map of $E_\infty$-ring spectra. There is an equivalence of $B$-modules \[L_{B/A}\simeq QI_0(B\sma_A B).\]
\end{cor}
This is analogous to what Basterra \cite{basterra} adopted as a definition for $L_{B/A}$ in a model-categorical setting. That was the first published definition. The approach from (\ref{eq-deftaq}) had been used in the preprint \cite{kriz}, albeit formulated in a different language. Only in \cite{basterra-mandell} were the two approaches first proven to be equivalent. %

\section{The cotangent complex of Thom spectra} \label{section-thom}

In this section we will prove the main result, \cref{thm-taqthom}, giving an expression for the cotangent complex of Thom $E_\infty$-algebras.\\

An \emph{$E_\infty$-monoid} is a commutative algebra object in $\S$ in the sense of \cite[2.1.3.1]{ha}, or, equivalently, a commutative monoid object in $\S$ in the sense of \cite[2.4.2.1]{ha} (i.e. a special $\Gamma$-space). They form an $\infty$-category $\Mon(\S)$. If an $E_\infty$-monoid $M$ is grouplike, i.e. if the monoid $\pi_0(M)$ is a group, we say $M$ is an \emph{$E_\infty$-group}. These form an $\infty$-category denoted by $\Grp(\S)$.

Let $R$ be an $E_\infty$-ring spectrum. An $R$-module $M$ is \emph{invertible} if there exists an $R$-module $N$ such that $M\sma_R N\simeq R$. We let $\Pic(R)$ be the \emph{Picard space} of $R$: this is the core (i.e. the maximal subspace) of the full subcategory of $\Mod_R$ on the invertible $R$-modules. 

If $Z$ is a space and $f:Z\to \Pic(R)$ is a map of spaces, the \emph{Thom $R$-module} of $f$ is defined as
\[Mf\coloneqq \colim(\xymatrix{Z\ar[r]^-f & \Pic(R) \ar@{^(->}[r] & \Mod_R}).\]
This defines a functor $M:\S_{/\Pic(R)} \to \Mod_R$.

As noted in \cite[7.7]{abg}, \cite[Section 3]{barthel-antolin}, $\Pic(R)$ is an $E_\infty$-group. If $G$ is an $E_\infty$-monoid and $f:G\to \Pic(R)$ is an $E_\infty$-map, then $Mf$ becomes an $E_\infty$-$R$-algebra \cite[8.1]{abg}, \cite[3.2]{barthel-antolin}. When $f=\{R\}$ is the constant map at $R\in \Pic(R)$, then $Mf\simeq R\sma \bS[G]$ as $E_\infty$-$R$-algebras.%

\subsection{The main result}
We will need the following proposition.

\begin{prop} \label{Proposition} %
Consider $\Sigma^\infty \Omega^\infty:\Sp\to \Sp$. The counit natural transformation
\[\Sigma^\infty \Omega^\infty \Rightarrow \id_{\Sp}\]
exhibits $\id_\Sp$ as the excisive approximation to $\Sigma^\infty\Omega^\infty$.
 
In particular, for a spectrum $X$, there is an equivalence
\begin{equation}
\label{eq-xcolim} X\simeq \colim_{\Sp} (\Sigma^\infty \Omega^\infty X \to \Omega \Sigma^\infty \Omega^\infty \Sigma X  \to \Omega^2\Sigma^\infty \Omega^\infty \Sigma^2 X \to \cdots) \end{equation}
natural in $X$. %
\end{prop}
\begin{proof}
We will first establish the natural equivalence (\ref{eq-xcolim}). We will then observe that this equivalence is obtained from the counit $\Sigma^\infty \Omega^\infty X \to X$, in such a way that the main assertion will have been proven, by \ref{section-good}.(\ref{item-colim}).

Let $X$ be a spectrum. First, we prove there is a natural equivalence of functors 
 \begin{align*}
 \Map_{\Sp}(X, -) \xrightarrow{ \simeq  } \lim_{n\in \N} \Map_{\S_*}(\Omega^\infty \Sigma^n X, \Omega^\infty \Sigma^n - )
\end{align*}
where the arrows in the sequential limit are loop functors. We do this carefully, taking care of naturality: indeed, it is easier to see that the two functors are objectwise equivalent using that $\Sp \simeq \lim( \cdots \xrightarrow{ \Omega } \S_* \xrightarrow{ \Omega } \S_*)$ \cite[1.4.2.24]{ha}, writing a mapping space as a pullback and commuting pullbacks with limits.

By \cite[2.2.1.2]{htt}, it suffices to establish an equivalence of the corresponding left fibrations.
The left hand side corresponds to the left fibration $\Sp_{X/} \to \Sp$. 
For the right hand side, first observe that for $n\in \N$, the functor $\Map_{\S_*}(\Omega^\infty \Sigma^n X, - )$ corresponds to the left fibration 
$(\S_*)_{\Omega^\infty \Sigma^n X /} \to \S_*$. Pulling back this left fibration along the functor $\Omega^\infty \Sigma^n: \Sp \to \S_*$,
\begin{center}
 \begin{tikzcd}
  \Sp \times_{\S_*} (\S_*)_{\Omega^\infty \Sigma^n X /} \arrow[r] \arrow[d, twoheadrightarrow] & (\S_*)_{\Omega^\infty \Sigma^n X/} \arrow[d, twoheadrightarrow] \\
  \Sp \arrow[r, "\Omega^\infty \Sigma^n"] & \S_*
 \end{tikzcd}
\end{center}
gives us the left fibration that classifies the functor $\Map_{\S_*}(\Omega^\infty \Sigma^n X, \Omega^\infty \Sigma^n - )$.
Taking limits, it follows that the functor $\lim_n \Map_{\S_*}(\Omega^\infty \Sigma^n X, \Omega^\infty \Sigma^n - )$ 
corresponds to the left fibration 
$\lim_n (\Sp \times_{\S_*} (\S_*)_{\Omega^\infty \Sigma^n X /}) \simeq \Sp \times_{\lim_{\mathbb{N}}\S_*} \lim_n(\S_*)_{\Omega^\infty \Sigma^n X /}$.

By the Yoneda lemma, a map of left fibrations $\Sp_{X/} \to  \Sp \times_{\lim_{\mathbb{N}}\S_*} \lim_n(\S_*)_{\Omega^\infty \Sigma^n X /}$ is 
uniquely determined by a choice of object in the fiber of $\Sp \times_{\lim_{\mathbb{N}}\S_*} \lim_n(\S_*)_{\Omega^\infty \Sigma^n X /} \to \Sp$
over $X$, which is given by the space $\{X\} \times \lim_n \Map_{\S_*}(\Omega^\infty \Sigma^n X, \Omega^\infty \Sigma^n X)$. Thus the object
$(X, (\id_{\Omega^\infty \Sigma^n X})_{n\in \N})$ in the fiber induces a map of left fibrations over $\Sp$
$$I: \Sp_{X/} \to  \Sp \times_{\lim_{\mathbb{N}}\S_*} \lim_n(\S_*)_{\Omega^\infty \Sigma^n X /}.$$
We want to prove this map is an equivalence. 
Note we can recover the right hand side directly as the following pullback: 
 \begin{center}
 \begin{tikzcd}
  \Sp \times_{\lim_\mathbb{N} \S_*} \lim_n (\S_*)_{\Omega^\infty \Sigma^n X/ } \arrow[d, "\pi_1", twoheadrightarrow] %
  \arrow[r, "\simeq", "\pi_2"'] & \lim_n (\S_*)_{\Omega^\infty \Sigma^n X/ } \arrow[d, twoheadrightarrow] \\
  \Sp \arrow[r, "\simeq"]  & \lim_\mathbb{N} \S_*.
 \end{tikzcd}
 \end{center}
The bottom equivalence is $\Sp \simeq \lim( \cdots \xrightarrow{ \Omega } \S_* \xrightarrow{ \Omega } \S_*) = \lim_\mathbb{N} \S_*$, from \cite[1.4.2.24]{ha}. 
This implies that the top horizontal map is an equivalence as well 
\cite[3.3.1.3]{htt}. 
Thus, in order to prove that $I$ is an equivalence, by $2$-out-of-$3$ it suffices to show that
$$ \pi_2 \circ I: \Sp_{X/} \to \lim_n(\S_*)_{\Omega^\infty\Sigma^n X/}$$
is an equivalence. 

First, we will construct an equivalence $\Sp_{X/}\xrightarrow{\simeq} \lim_n(\S_*)_{\Omega^\infty\Sigma^n X/}$. Then we will observe it is indeed $\pi_2 \circ I$.
We have 
\begin{align*}
 \Sp_{X/} \simeq \Sp^{\Delta^1} \times_{\Sp} \Delta^0 \xrightarrow{  \simeq  } (\lim_{\mathbb{N}} \S_*)^{\Delta^1} \times_{\lim_{\mathbb{N}} \S_*} \Delta^0.
\end{align*}
Using that limits commute with pullbacks and exponentials $(-)^{\Delta^1}$, 
\begin{align*}
 \lim_{\mathbb{N}}((\S_*)^{\Delta^1} \times_{\S_*} \Delta^0) %
 \simeq \lim_n(\S_*)_{\Omega^\infty\Sigma^n X/}
\end{align*}
gives us an equivalence. Notice this equivalence takes the object $\id_X$ in $\Sp_{X/}$ to $(\id_{\Omega^\infty \Sigma^n X})_{n \in \mathbb{N}}$ 
and thus, by the Yoneda lemma, is equivalent to $\pi_2 \circ I$, as they both take $\id_X$ to the same object. 
This gives us the desired result.

Next, we have a natural equivalence 
\begin{align*}
 \lim_n\Map_{\S_*}(\Omega^\infty \Sigma^n X, \Omega^\infty \Sigma^n - ) \xrightarrow{\simeq} \lim_n\Map_{\Sp}(\Omega^n \Sigma^\infty \Omega^\infty \Sigma^n X, -)
\end{align*}
induced by the adjunction $(\Sigma^\infty, \Omega^\infty)$. 
Finally, we have a natural equivalence 
\begin{align*}
 \lim_n\Map_{\Sp}(\Omega^n \Sigma^\infty \Omega^\infty \Sigma^n X, -) \xrightarrow{ \simeq } \Map_{\Sp}(\colim_n \Omega^n \Sigma^\infty \Omega^\infty \Sigma^n X, -).
\end{align*}
 Combining all three gives us a natural equivalence 
\begin{align*}
 \Map_{\Sp}(X, -) \xrightarrow{ \simeq } \Map_{\Sp}(\colim_n \Omega^n \Sigma^\infty \Omega^\infty \Sigma^n X, -)
\end{align*}
which by the Yoneda lemma is induced by a map $\colim_n \Omega^n \Sigma^\infty \Omega^\infty \Sigma^n X \to X$ and 
is explicitly given by the image of the identity map $\id_X$ under this natural equivalence. 
By inspection, the map $\colim_n \Omega^n \Sigma^\infty \Omega^\infty \Sigma^n X \to X$ can be characterized as the unique map 
that comes from the cocone given by $\Omega^n \Sigma^\infty \Omega^\infty \Sigma^n X \to \Omega^n \Sigma^n X \xrightarrow{\simeq} X$ where the first arrow is induced by the counit of the $(\Sigma^\infty,\Omega^\infty)$ adjunction. %
\end{proof}

Let us fix some notation necessary for the statement of the following theorem and its proof. Let $B^\infty:\Grp(\S)\to \Spc$ denote the standard equivalence between the $\infty$-categories of $E_\infty$-groups and of connective spectra \cite[5.2.6.26]{ha}. Let $B:\Grp(\S)\to \Grp(\S)$ denote the bar construction functor, which can be defined in this context simply as the functor that takes $G$ to the pushout $*\leftarrow G \to *$. We will iterate this functor, getting $B^n:\Grp(\S)\to \Grp(\S)$ for $n\geq 1$.

If we take an $E_1$-group instead of an $E_\infty$-group as input, then we can extend the bar construction to a functor $\Bar:\Grpp{1}(\S)\to \S_*$ \cite[5.2.2]{ha}. It agrees with $BG$ whenever this makes sense: namely, if $G\in \Grp(\S)$, then the pointed space underlying $BG$ is equivalent to $\Bar$ of the $E_1$-group underlying $G$; this follows from the remarks in \cite[5.2.2.3/4]{ha}.

There are also iterated bar construction functors $\Bar^{(n)}:\Grpp{n}(\S)\to \S_*^{\geq n}$ for $n\geq 1$ taking an $E_n$-group and giving a pointed $(n-1)$-connected space, obtained from iterations of $\Bar$ \cite[Page 802, 5.2.3]{ha}. These $\Bar^{(n)}$ are equivalences \cite[5.2.6.10(3)]{ha}, and coincide with $B^n$ after judicious forgetting of structure. %

\begin{thm} \label{thm-taqthom} Let $R$ be an $E_\infty$-ring spectrum. Let $f:G\to \Pic(R)$ be a map of $E_\infty$-groups. There is an equivalence of $Mf$-modules
\[L_{Mf/R} \simeq Mf \sma B^\infty G.\]
\begin{proof} By \cite[4.11]{rsv-thom}, $S^n\otimes_R Mf \simeq Mf \sma \bS[B^nG]$ in $\CAlg_\under{Mf}{Mf}$. %
Now, note that
\begin{equation}\label{eq-otilde-mf}
S^n\otilde_R Mf\simeq Mf \sma \Sigma^\infty B^n G
\end{equation}
as $Mf$-modules. Indeed, since $\xymatrix{S^0\simeq \ast_+ \ar[r]^-{(x_0)_+} & (B^n G)_+ \ar[r] & B^nG}$ is a cofiber sequence in $\S_*$ where $x_0$ denotes the basepoint of $B^nG$, then applying $Mf\sma \Sigma^\infty(-)$ we get a cofiber sequence in $\Mod_{Mf}$
\[Mf \simeq Mf \sma \Sigma^\infty_+(*) \to Mf \sma \Sigma^\infty_+B^nG \to Mf \sma \Sigma^\infty B^n G,\]
whereas by the definition of $J$ in \cref{lemma-ij} we get the equivalence (\ref{eq-otilde-mf}).

By \cref{prop-colim} we get equivalences of $Mf$-modules
\begin{align} 
\nonumber L_{Mf/R}
&\simeq \colim_{\Mod_{Mf}}(\xymatrix{Mf \sma \Sigma^\infty G \ar[r] & \Omega (Mf \sma  \Sigma^\infty BG) \ar[r] & \Omega^2 (Mf \sma \Sigma^\infty B^2G) \ar[r] & \cdots }) \\
\label{eq-mfcolim}&\simeq Mf \sma\colim_{\Sp}(\xymatrix{\Sigma^\infty G \ar[r] & \Omega \Sigma^\infty BG \ar[r] & \Omega^2 \Sigma^\infty B^2G \ar[r] & \cdots })
\end{align}
where in the second line we have used the stability of $\Sp$ and $\Mod_{Mf}$ together with the fact that $Mf\sma -:\Sp \to \Mod_{Mf}$ commutes with colimits; note that $\Omega$ now denotes the loop functor in $\Sp$.  %

We will now prove that $B^\infty G$ is equivalent to the colimit in (\ref{eq-mfcolim}), which we can rewrite as
\begin{equation}\label{eq-mfcolim2}
\colim_{\Sp} (\Sigma^\infty U_0 G \to \Omega\Sigma^\infty i_1\Bar (U_1G)  \to \Omega^2\Sigma^\infty i_2 \Bar^{(2)}(U_2G) \to \cdots)
\end{equation}
where $U_n:\Grp(\S) \to \Grpn(\S)$ is the forgetful functor, $E_0$-groups are pointed spaces, and $i_n:\S_*^{\geq n}\to \S_*$ is the inclusion functor. Now note that $B^\infty$ is constructed as the limit of the $\Bar^{(n)}$ as follows \cite[5.2.6.26]{ha} (the functors $\beta_n$ in that reference are the inverses of $\Bar^{(n)}$ \cite[5.2.6.10(3)]{ha})
\[\xymatrix{
\Spc \simeq \lim(  \cdots \ar[r]^-\Omega & \S_*^{\geq n+1}\ar[r]^\Omega & \S_*^{\geq n} \ar[r]^-\Omega & \cdots \ar[r]^-\Omega & \S_*^{\geq 1} \ar[r]^-\Omega & \S_*) \\
\Grp(\S)\ar@<1cm>[u]^-{B^\infty}_-\sim \simeq \lim (\cdots \ar[r] & \Grpp{n+1}(\S) \ar[u]^-{\Bar^{(n+1)}} \ar[r] & \Grpp{n}(\S) \ar[r] \ar[u]^-{\Bar^{(n)}} &  \cdots \ar[r] & \Grpp{1}(\S) \ar[r] \ar[u]^-\Bar & \S_* \ar[u]^-\id),
}\]
so in particular we have a commutative diagram
\[\xymatrix{\Spc \ar[r] & \S_*^{\geq n}  \\ \Grp(\S) \ar[u]^-{B^\infty} \ar[r]_{U_n} & \Grpn(\S) \ar[u]_-{\Bar^{(n)}} }\]
where the top horizontal functor is the projection to the corresponding term of the limit. After composing it with $i_n$ this functor becomes $\Omega^{\infty-n}=\Omega^\infty \Sigma^n:\Spc\to \S_*$, the $n$-th space functor. %
Therefore, (\ref{eq-mfcolim2}) is equivalent to
\[\colim_{\Sp} (\Sigma^\infty \Omega^\infty B^\infty G \to \Omega\Sigma^\infty \Omega^\infty \Sigma B^\infty G  \to \Omega^2\Sigma^\infty \Omega^\infty \Sigma^2 B^\infty G \to \cdots).\]
This is equivalent to $B^\infty G$ by \cref{Proposition}.
\end{proof}
\end{thm}

\begin{rmk} A model-categorical version of \cref{thm-taqthom} first appeared in \cite[Corollary, Page 907]{basterra-mandell}. Their result, however, only applies to Thom spectra of maps to $BGL_1(\bS)$, whereas ours applies to maps to $\Pic(R)$ where $R$ is any $E_\infty$-ring spectrum. Already considering maps into $\Pic(\bS)$ leads to interesting, non-connective examples, as we shall now see.
\end{rmk}

\begin{ex} \label{ex-mu-mup} Let $MU$ denote the complex cobordism spectrum, which is the Thom $E_\infty$-ring spectrum of the complex $J$-homomorphism $BU\to \Pic(\bS)$. Then \[L_{MU}\simeq MU\sma bu\]
as $MU$-modules, where $bu$ denotes $B^\infty BU$. This recovers \cite[Corollary, Page 907]{basterra-mandell}. Similarly, let $MUP$ denote the periodic complex cobordism spectrum, which is the Thom $E_\infty$-ring spectrum of the complex $J$-homomorphism $BU\times \Z\to \Pic(\bS)$. Then
\[L_{MUP} \simeq MUP \sma ku\]
as $MUP$-modules, where $ku\simeq B^\infty (BU \times \Z)$ is the connective complex $K$-theory spectrum.
\end{ex}

\subsection{The cotangent complex as a Thom spectrum}
If $G$ is not only an $E_\infty$-group but also an \emph{$E_\infty$-ring space} in the sense of \cite[7.1]{ggn}, then $B^\infty G$ is a connective $E_\infty$-ring spectrum. Many interesting examples of $E_\infty$-ring spaces arise as group completions of \emph{$E_\infty$-rig spaces} (similar to $E_\infty$-ring spaces but now the underlying additive structure does not necessarily admit inverses). For example, $\bigsqcup_{n\geq 0} B\Sigma_n$, $\bigsqcup_{n\geq 0}BU(n)$ and $\bigsqcup_{n\geq 0} BGL_n(A)$ where $A$ is a commutative ring %
are examples of $E_\infty$-rig spaces: their corresponding connective $E_\infty$-ring spectra are $\bS$, $ku$ and the algebraic $K$-theory $K(A)$ respectively. See \cite[Sections 7/8]{ggn} for more details.

\begin{prop} \label{prop-ringspace} Let $R$ be an $E_\infty$-ring spectrum, let $G$ be an $E_\infty$-ring space and let $f:G\to \Pic(R)$ be an $E_\infty$-map (with respect to the $E_\infty$-group structure of $G$). Then the $R$-module $L_{Mf/R}$ underlies a Thom $E_\infty$-$(R\sma B^\infty G)$-algebra.%
\end{prop}

\begin{proof}
First, note that if $R\to T$ is a map of $E_\infty$-ring spectra, then extension of scalars restricts to an $E_\infty$-map $-\sma_R T:\Pic(R)\to \Pic(T)$ making the following diagram commute:
 \[\xymatrix{
 \Pic(R) \ar@{^(->}[r] \ar[d]_-{-\sma_R T} & \Mod_R \ar[d]^-{-\sma_R T} \\ 
 \Pic(T) \ar@{^(->}[r] & \Mod_T.
 }\]
 Indeed, the functor $-\sma_RT:\Mod_R\to \Mod_T$ takes invertible $R$-modules to invertible $T$-modules.
 
 Now recall that, as an $R$-module, $Mf\simeq \colim(G\xrightarrow{f} \Pic(R) \to \Mod_R)$. Let $T=R\sma B^\infty G$. Since $-\sma_R T:\Mod_R \to \Mod_T$ preserves colimits, we get equivalences of $T$-modules
 \begin{align*}
 Mf \sma B^\infty G &\simeq Mf \sma_R T \\
& \simeq \colim( \xymatrix{G \ar[r]^-f & \Pic(R) \ar[r] & \Mod_R \ar[r]^-{-\sma_R T} & \Mod_T} ) \\
  &\simeq \colim( \xymatrix{G \ar[r]^-f & \Pic(R) \ar[r]^-{-\sma_R T} & \Pic(T) \ar[r] & \Mod_T}) \\
  &=M((-\sma_R T) \circ f).
 \end{align*}
Since we proved in \cref{thm-taqthom} that $Mf\sma B^\infty G\simeq L_{Mf/R}$ and $(-\sma_R T) \circ f$ is an $E_\infty$-map, 
this proves the result.
\end{proof}
\begin{rmk} This echoes with the result that the factorization homology of a Thom $E_n$-algebra is a Thom spectrum \cite[4.2]{klang-thom}, or with the related result that the tensor of a Thom $E_\infty$-algebra with a space is again a Thom $E_\infty$-algebra \cite[4.10]{rsv-thom}.
\end{rmk}

\begin{ex}
Continuing the example of $MUP$ from \cref{ex-mu-mup}, if we take $f$ to be the complex $J$-homomorphism $BU\times \Z\to \Pic(\bS)$, then $L_{MUP}\simeq MUP\sma ku$ underlies a Thom $E_\infty$-$ku$-algebra, even if $ku$ is not the Thom spectrum of any $E_3$-map from an $E_3$-group \cite{ahl-ku-thom}.
\end{ex}

In \cref{prop-ringspace}, we proved that $L_{Mf/R}$ is the Thom module of an $E_\infty$-map provided $G$ is an $E_\infty$-ring space. The hypothesis was not superfluous: let us now give an example of a cotangent complex of a Thom $E_\infty$-ring spectrum which cannot be the Thom module of an $E_\infty$-map, simply because it would then be an $E_\infty$-ring spectrum and this cannot happen in this example:
\begin{ex}
Let $G$ be an abelian group of infinite order such that every element has finite order. For example, one could take $\Q/\Z$ or an infinite direct sum of cyclic groups of arbitrarily large order.  %
 \par 
 Let $f: G \to \Pic(\bS)$ be the the constant map at $\bS$. %
 Let us prove that $L_{Mf}$ is not the underlying spectrum of an $E_\infty$-ring spectrum. 
 Such a structure would induce a ring structure on $\pi_0(L_{Mf})$, so it suffices to prove that the latter has no ring structure extending its abelian group structure.
 \par 
 Before we prove this, let us use \cref{thm-taqthom} to compute $L_{Mf}$. We have equivalences of spectra
 $$L_{Mf} \simeq Mf \sma B^\infty G \simeq \Sigma^\infty_+ G \sma HG \simeq \bigvee_{G} \bS \sma HG \simeq \bigvee_G HG$$
 where we used the fact that $B^\infty G\simeq HG$ is the Eilenberg--Mac Lane spectrum of $G$, and that since $G$ is a discrete group then $\Sigma^\infty_+ G$ is a wedge of sphere spectra. 
 In particular, $\pi_0(L_{Mf}) = \bigoplus_G G$. 
 \par 
 
Now, note that $G'=\bigoplus_G G$ is again an abelian group of infinite order such that every element has finite order. This implies that it is not the underlying abelian group of a ring. If it were, the multiplicative unit would have finite order, and this number would bound the order of all the other elements, contradicting that $G'$ has infinite order.
\end{ex}

\begin{rmk} 
As we just observed, in general $L_{Mf/R}$ is not the Thom module of an $E_\infty$-map. It is, however, the colimit of iterated loops of Thom modules of $E_\infty$-maps by (\ref{eq-mfcolim}), since by \cite[4.8]{rsv-thom}
 \[Mf \sma \Sigma^\infty B^nG %
 \simeq M(G \times B^nG \xrightarrow{ \pi_1} G \xrightarrow{f} \Pic(R))\]
 and so 
 \[L_{Mf/R} \simeq \colim_{\Mod_{Mf}} ( \Omega^nM(G \times B^nG \xrightarrow{ \pi_1 } G \xrightarrow{f} \Pic(R))).\]
\end{rmk}

\section{Étale extensions and solid ring spectra} \label{section-etale}

We will now extend the results of the previous section to compute cotangent complexes of two different types of extensions of Thom $E_\infty$-algebras.\\

Following \cite[7.5.0.1/2/4]{ha}, a map of (ordinary) commutative rings $A\to B$ is \emph{étale} if $B$ is finitely presented as an $A$-algebra, $B$ is flat as an $A$-module, and there exists an idempotent element $e\in B\otimes_A B$ such that the multiplication map $B\otimes_A B\to B$ induces an isomorphism $(B\otimes_A B)[e^{-1}] \cong B$.  If $A\to B$ is a map of $E_\infty$-ring spectra, it is \emph{étale} if $\pi_0(A)\to \pi_0(B)$ is étale and $B$ is flat as an $A$-module, i.e. the natural map \[\pi_*(A)\otimes_{\pi_0(A)} \pi_0(B) \to \pi_*(B)\] is an isomorphism.

\begin{prop} \label{prop-cotetale} If $A\to B$ is an étale map of $E_\infty$-ring spectra, then $L_{B/A}$ vanishes. %
Therefore, if $R$ is an $E_\infty$-ring spectrum and $A\to B$ is a map of $E_\infty$-$R$-algebras which is étale, there is an equivalence of $B$-modules
\[L_{B/R}\simeq B\sma_A L_{A/R}.\]
\begin{proof}
The first statement is \cite[7.5.4.5]{ha}. Note that Lurie adds a connectivity hypothesis, but it is not used in the proof. %
The second statement now follows from the transitivity cofiber sequence \cite[7.3.3.6]{ha}
\[B\sma_A L_{A/R} \to L_{B/R} \to L_{B/A}. \qedhere\]
\end{proof}
\end{prop}

\begin{cor} \label{cor-taqetale} Let $R$ be an $E_\infty$-ring spectrum. Let $f:G\to \Pic(R)$ be a map of $E_\infty$-groups. Let $Mf\to B$ be a map of $E_\infty$-$R$-algebras which is étale. There is an equivalence of $B$-modules
\[L_{B/R} \simeq B\sma B^\infty G.\]
\begin{proof}
By \cref{prop-cotetale}, $L_{B/R}\simeq B\sma_{Mf} L_{Mf/R}$ as $B$-modules, and by \cref{thm-taqthom} we have an equivalence of $Mf$-modules $L_{Mf/R}\simeq Mf \sma B^\infty G$. Putting these together finishes the proof.
\end{proof}
\end{cor}

\begin{ex} \label{ex-inverting at zero} Let $R$ be an $E_\infty$-ring spectrum and $x\in \pi_0(R)$. By \cite[7.5.0.6/7]{ha}, %
there is an étale map of $E_\infty$-ring spectra $R\to R[x^{-1}]$ which universally inverts the homotopy element $x$.  %
We deduce from \cref{prop-cotetale} that there is an equivalence of $R[x^{-1}]$-modules
\[L_{R[x^{-1}]} \simeq (L_R)[x^{-1}].\]
\end{ex}

There are interesting instances where we want to invert a homotopy element $x$ that is not in degree $0$, as we shall see below.  Unfortunately, in this case the map $R\to R[x^{-1}]$ may not be étale: indeed, it may not be flat. For example, if $R$ is connective then any flat $R$-module is necessarily connective, as follows from the definition \cite[7.2.2.11]{ha}.

To remedy this, we recall the notion of solidity: An $E_\infty$-$A$-algebra $B$ is \emph{solid} if the multiplication map $\mu: B \sma_A B \to B$ is an equivalence. %
Note in this case $B\sma_A B\simeq B$ as objects of $\CAlg_\BB$.

\begin{prop} \label{prop-cotsolid} If $A$ is an $E_\infty$-ring spectrum and $B$ is a solid $E_\infty$-$A$-algebra, then $L_{B/A}$ vanishes. %
Therefore, if $R$ is an $E_\infty$-ring spectrum, $A$ an $E_\infty$-$R$-algebra and $B$ a solid $E_\infty$-$A$-algebra, then there is an equivalence of $B$-modules
\[L_{B/R}\simeq B\sma_A L_{A/R}.\]
\begin{proof}
For the first statement, consider the equivalences  
$$L_{B/A} \simeq (P_1I)(B \sma_A B) \simeq (P_1I)(B) $$
where the first equivalence is (\ref{eq-p1}) and the second follows from solidity. 
Now $(P_1I)(B)$ is trivial, since $B\in \CAlg_\BB$ is a zero object and $P_1I$ is reduced (see \ref{section-good}.(\ref{item-colim})).

The second statement now follows from the transitivity cofiber sequence \cite[7.3.3.6]{ha}
\[B\sma_A L_{A/R} \to L_{B/R} \to L_{B/A}. \qedhere\]
\end{proof}
\end{prop}

The following corollary is proven analogously to \cref{cor-taqetale}.

\begin{cor} \label{cor-taqsolid} Let $R$ be an $E_\infty$-ring spectrum. Let $f:G\to \Pic(R)$ be a map of $E_\infty$-groups. Let $B$ be a solid $E_\infty$-$Mf$-algebra. There is an equivalence of $B$-modules
\[L_{B/R} \simeq B\sma B^\infty G.\]
\end{cor}

\begin{ex}\label{ex-so}
 Let $R$ be an $E_\infty$-ring spectrum and $x \in \pi_*(R)$. By \cite[4.3.17]{lurie-chromatic2}, there exists an $E_\infty$-$R$-algebra $R[x^{-1}]$ where $x$ has been universally inverted. Now note that %
 $R[x^{-1}]$ is a 
  solid $E_\infty$-$R$-algebra. Indeed, by \cite[7.3]{rsv-thom}, we have 
 \[R[x^{-1}] \sma_R R[x^{-1}] \simeq R[x^{-1}][x^{-1}] \simeq R[x^{-1}].\]
Therefore, by \cref{prop-cotsolid}, %
we have an equivalence
\[L_{R[x^{-1}]} %
\simeq (L_R)[x^{-1}].\]
This generalizes  \cref{ex-inverting at zero}, where $x$ was only allowed to be in degree zero.
\end{ex}

\begin{ex}Let $KU$ denote the periodic complex topological $K$-theory $E_\infty$-ring spectrum. Snaith \cite{snaith79}, \cite{snaith81} proved that $KU\simeq \bS[K(\Z,2)][x^{-1}]$ as homotopy commutative ring spectra (i.e. commutative monoids in the homotopy category of spectra), where $x\in \pi_2 \bS[K(\Z,2)]$ is induced by the fundamental class in $K(\Z,2)$. See \cite[6.5.1]{lurie-chromatic2} for one improvement of such an equivalence to an equivalence of $E_\infty$-ring spectra.

Since $\bS[K(\Z,2)]\simeq M(K(\Z,2) \xrightarrow{\{\bS\}} \Pic(\bS))$, we can apply \cref{cor-taqsolid} and \cref{ex-so} to deduce that $L_{KU}\simeq KU \sma \Sigma^2 H\Z$. Recall that the inclusion $\Z\to \Q$ induces an equivalence $KU\sma H\Z\simeq KU \sma H\Q$ \cite[16.25]{switzer}. Combining this result with Bott periodicity, we obtain:
\[L_{KU}\simeq KU\sma H\Q,\]
the rationalization of $KU$, a result first gotten in \cite[8.4]{stonek-thhku} %
in a model-categorical context.
\end{ex}

\bibliographystyle{alpha}
\bibliography{main}

\newcommand{\etalchar}[1]{$^{#1}$}
\begin{thebibliography}{ABG{\etalchar{+}}14}

\bibitem[ABG{\etalchar{+}}14]{abghr-infty}
Matthew Ando, Andrew~J. Blumberg, David Gepner, Michael~J. Hopkins, and Charles
  Rezk.
\newblock An {$\infty$}-categorical approach to {$R$}-line bundles,
  {$R$}-module {T}hom spectra, and twisted {$R$}-homology.
\newblock {\em J. Topol.}, 7(3):869--893, 2014.

\bibitem[ABG18]{abg}
Matthew Ando, Andrew~J. Blumberg, and David Gepner.
\newblock Parametrized spectra, multiplicative {T}hom spectra and the twisted
  {U}mkehr map.
\newblock {\em Geom. Topol.}, 22(7):3761--3825, 2018.

\bibitem[ACB19]{barthel-antolin}
Omar Antol\'{\i}n-Camarena and Tobias Barthel.
\newblock A simple universal property of {T}hom ring spectra.
\newblock {\em J. Topol.}, 12(1):56--78, 2019.

\bibitem[AHL09]{ahl-ku-thom}
Vigleik Angeltveit, Michael~A. Hill, and Tyler Lawson.
\newblock The spectra {$ko$} and {$ku$} are not {T}hom spectra: an approach
  using {$THH$}.
\newblock In {\em New topological contexts for {G}alois theory and algebraic
  geometry ({BIRS} 2008)}, volume~16 of {\em Geom. Topol. Monogr.}, pages 1--8.
  Geom. Topol. Publ., Coventry, 2009.

\bibitem[And67]{andre}
Michel Andr\'{e}.
\newblock {\em M\'{e}thode simpliciale en alg\`ebre homologique et alg\`ebre
  commutative}.
\newblock Lecture Notes in Mathematics, Vol. 32. Springer-Verlag, Berlin-New
  York, 1967.

\bibitem[Bas99]{basterra}
M.~Basterra.
\newblock Andr\'e-{Q}uillen cohomology of commutative {$S$}-algebras.
\newblock {\em J. Pure Appl. Algebra}, 144(2):111--143, 1999.

\bibitem[Bec03]{beck}
Jonathan~Mock Beck.
\newblock Triples, algebras and cohomology.
\newblock {\em Repr. Theory Appl. Categ.}, (2):1--59, 2003.

\bibitem[BM05]{basterra-mandell}
Maria Basterra and Michael~A. Mandell.
\newblock Homology and cohomology of {$E_\infty$} ring spectra.
\newblock {\em Math. Z.}, 249(4):903--944, 2005.

\bibitem[BM11]{basterra-mandell-En}
Maria Basterra and Michael~A. Mandell.
\newblock Homology of {$E_n$} ring spectra and iterated {$THH$}.
\newblock {\em Algebr. Geom. Topol.}, 11(2):939--981, 2011.

\bibitem[Cis19]{cisinski}
Denis-Charles Cisinski.
\newblock {\em Higher categories and homotopical algebra}, volume 180 of {\em
  Cambridge Studies in Advanced Mathematics}.
\newblock Cambridge University Press, Cambridge, 2019.

\bibitem[GGN15]{ggn}
David Gepner, Moritz Groth, and Thomas Nikolaus.
\newblock Universality of multiplicative infinite loop space machines.
\newblock {\em Algebr. Geom. Topol.}, 15(6):3107--3153, 2015.

\bibitem[Kla18]{klang-thom}
Inbar Klang.
\newblock The factorization theory of {T}hom spectra and twisted nonabelian
  {P}oincar\'{e} duality.
\newblock {\em Algebr. Geom. Topol.}, 18(5):2541--2592, 2018.

\bibitem[Kri93]{kriz}
Igor Kriz.
\newblock Towers of {$E_\infty$}-ring spectra with an application to {B}{P},
  1993.
\newblock Preprint.

\bibitem[Kuh07]{kuhn-goodwillie}
Nicholas~J. Kuhn.
\newblock Goodwillie towers and chromatic homotopy: an overview.
\newblock In {\em Proceedings of the {N}ishida {F}est ({K}inosaki 2003)},
  volume~10 of {\em Geom. Topol. Monogr.}, pages 245--279. Geom. Topol. Publ.,
  Coventry, 2007.

\bibitem[Lur09]{htt}
Jacob Lurie.
\newblock {\em Higher topos theory}, volume 170 of {\em Annals of Mathematics
  Studies}.
\newblock Princeton University Press, Princeton, NJ, 2009.

\bibitem[Lur17]{ha}
Jacob Lurie.
\newblock Higher algebra.
\newblock \url{http://www.math.ias.edu/~lurie/papers/HA.pdf}, September 2017.

\bibitem[Lur18a]{lurie-chromatic2}
Jacob Lurie.
\newblock Elliptic {C}ohomology {I}{I}: {O}rientations.
\newblock \url{http://www.math.ias.edu/~lurie/papers/Elliptic-II.pdf}, April
  2018.

\bibitem[Lur18b]{sag}
Jacob Lurie.
\newblock Spectral algebraic geometry.
\newblock \url{https://www.math.ias.edu/~lurie/papers/SAG-rootfile.pdf},
  February 2018.

\bibitem[Qui68]{quillen-homology}
Daniel Quillen.
\newblock Homology of commutative rings.
\newblock 1968.
\newblock Mimeographed notes, MIT.

\bibitem[RSV19]{rsv-thom}
Nima Rasekh, Bruno Stonek, and Gabriel Valenzuela.
\newblock Thom spectra, higher {$THH$} and tensors in $\infty$-categories,
  2019.
\newblock \href{https://arxiv.org/abs/1911.04345v3}{arXiv:1911.04345v3}.

\bibitem[Sch11]{schlichtkrull-higher}
Christian Schlichtkrull.
\newblock Higher topological {H}ochschild homology of {T}hom spectra.
\newblock {\em J. Topol.}, 4(1):161--189, 2011.

\bibitem[Sna79]{snaith79}
Victor~P. Snaith.
\newblock Algebraic cobordism and {$K$}-theory.
\newblock {\em Mem. Amer. Math. Soc.}, 21(221):vii+152, 1979.

\bibitem[Sna81]{snaith81}
Victor Snaith.
\newblock Localized stable homotopy of some classifying spaces.
\newblock {\em Math. Proc. Cambridge Philos. Soc.}, 89(2):325--330, 1981.

\bibitem[Sto20]{stonek-thhku}
Bruno Stonek.
\newblock Higher topological {H}ochschild homology of periodic complex
  {K}-theory.
\newblock {\em Topology Appl.}, 282:107302, 43, 2020.

\bibitem[Swi75]{switzer}
Robert~M. Switzer.
\newblock {\em Algebraic topology---homotopy and homology}.
\newblock Springer-Verlag, New York-Heidelberg, 1975.
\newblock Die Grundlehren der mathematischen Wissenschaften, Band 212.

\end{thebibliography}
\end{document}